\documentclass[11pt,oneside]{article}
\usepackage{amsfonts}
\usepackage{amssymb}
\usepackage{mathrsfs}
\usepackage{srcltx}
\usepackage{subfigure}
\usepackage{color}

\usepackage{amsmath,amsthm,amstext,amsopn,graphicx,cite,xcolor,enumitem}
\newtheorem{theorem}{Theorem}[section]
\newtheorem{lemma}[theorem]{Lemma}
\newtheorem{proposition}[theorem]{Proposition}
\newtheorem{corollary}[theorem]{Corollary}
\newtheorem{remark}[theorem]{Remark}

\theoremstyle{definition}

\newtheorem{example}[theorem]{Example}

\theoremstyle{remark}

\numberwithin{equation}{section}
\newcommand\bes{\begin{eqnarray}}
\newcommand\ees{\end{eqnarray}}
\newcommand{\bess}{\begin{eqnarray*}}
\newcommand{\eess}{\end{eqnarray*}}
\newcommand{\ba}{\begin{array}}
\newcommand{\ea}{\end{array}}
\newcommand{\f}{\frac}
\newcommand{\Om}{\Omega}

\newcommand{\var}{\varepsilon}

\newcommand{\lf}{\left}
\newcommand{\rr}{\right}

\newcommand{\dd}{\displaystyle}

\newcommand{\af}{\alpha}
\newcommand{\td}{\tilde}
\newcommand{\wtd}{\widetilde}
\newcommand\yy{\infty}
\newcommand\ttt{t\to\yy}

\newcommand{\ol}{\overline}
\newcommand{\nm}{\nonumber}
\newcommand{\rd}{{\rm d}}

\begin{document}
 \pagestyle{myheadings}


\date{}
\title{ \bf\large{Global stability of nonhomogeneous equilibrium solution for the diffusive Lotka-Volterra competition model}\footnote{Partially supported by a grant from China Scholarship Council, NSF Grant DMS-1715651, and NSFC Grant 11771110.}}
\author{Wenjie Ni\textsuperscript{1}\footnote{Current address: School of Science and Technology, University of New England, Armidale, NSW 2351, Australia},\ \ Junping Shi\textsuperscript{2}\footnote{Corresponding Author.},\ \ Mingxin Wang\textsuperscript{1}
 \\
{\small \textsuperscript{1} School of Mathematics, Harbin Institute of Technology, Harbin, 150001, China.\hfill{\ }}\\
{\small \textsuperscript{2} Department of Mathematics, College of William and Mary,  Williamsburg, VA, 23187-8795, USA.\setcounter{footnote}{-1}{}\footnote{{\it E-mails}: nwj1033159832@163.com (W.-J. Ni), jxshix@wm.edu (J.-P. Shi),
mxwang@hit.edu.cn (M.-X. Wang)}\hfill{\ }}\\}
\maketitle

\vspace{-6mm}
\begin{quote}
\noindent{\bf Abstract.} A diffusive Lotka-Volterra competition model is considered for the combined effect of spatial dispersal and spatial variations of resource on the population persistence and exclusion. First it is shown that in a two-species system in which the diffusion coefficients, resource functions and competition rates are all spatially heterogeneous, the positive equilibrium solution is globally asymptotically stable when it exists. Secondly the existence and  global asymptotic stability of the positive and semi-trivial equilibrium solutions are obtained for the model with arbitrary number of species under the assumption of spatially heterogeneous resource distribution. A new Lyapunov functional method is developed to prove the global stability of a non-constant equilibrium solution in heterogeneous environment.


\end{quote}

\section {Introduction}{\setlength\arraycolsep{2pt}

The uneven distribution of resources due to the effect of geological and environmental characteristics greatly enriches the diversity of ecosystems. In the past a few decades, the phenomenon of spatial heterogeneity of resources has attracted the attention of many researchers from both biology and mathematics, see \cite{cc2003,levin,lou2006,tilman1994competition}, for example. The dynamical properties (existence, bifurcation and local/global stability of nonconstant steady state) of mathematical models with spatial heterogeneity are more complicated. Simultaneously it brings an additional difficulty to study the global stability of nonconstant equilibrium solutions. 

In this paper we consider the global dynamics of the  diffusive Lotka-Volterra competition model of multi-species in a  nonhomogeneous environment:
\begin{equation}\label{a.1}
\begin{cases}
\dd \partial_t u_i=d_i(x)\Delta u_i+u_i\bigg(m_i(x)-\sum_{j=1}^k a_{ij}(x)u_j\bigg), & x\in \Omega,\; \ t>0,\; 1\leq i\leq k,\\
\dd \partial_\nu u_i=0,& x\in \partial\Omega,\;
 t>0,\; 1\leq i\leq k,\\
u_i(x,0)=\varphi_i(x)\geq, \, \not\equiv 0, & x\in\ol\Omega, \; 1\leq i\leq k,
 \end{cases}
\end{equation}
where $u_i(x,t)$ is the population density of $i$-th biological species, $m_i\in C^{\alpha}(\ol\Omega)$, $i=1,...,k$, represent the densities of non-uniform resources, and the nonnegative function $a_{ij}\in C^{\alpha}(\ol\Omega)$ is the strength of competition for species $u_i$ and $u_j$ at location $x$. Here $d_i(x)>0$ is the diffusion coefficient of $u_i$ at location $x$, the set $\Omega\subset\mathbb{R}^N$ is a bounded domain  with smooth boundary $\partial\Omega\in C^{2+\alpha}$, $\nu$ is the outward unit normal vector over $\partial\Omega$, and the homogeneous Neumann boundary condition indicates that this system is self-contained with zero population flux across the boundary.

If all of $d_i$, $m_i$ and $a_{ij}$ are positive constants (so the environment is spatially homogeneous), the global stability of positive constant equilibrium of \eqref{a.1} had been proved in the weak competition case, see \cite{bp1983, gbs1977} and the references therein. Lyapunov functional methods is used in proving the global stability of positive constant equilibrium of \eqref{a.1} in the homogenous case. However, the spatial heterogeneity of the environment may change the outcome of the competition, and it is an important biological question to understand how the spatially nonhomogeneous environment affects the competition between species. When $k=2$,  $d_i$ are constants, the two species $u_1$ and $u_2$} share the same spatially distributed resource $m_1(x)=m_2(x)$  and have the same competition coefficients $a_{ij}=1$,  it was shown in \cite{dhmp1998} that the species with  smaller diffusion coefficient survives while the other one with larger diffusion coefficient becomes extinct, that is, the slower diffuser prevails. The same question with $k\ge 3$ species remains as an open question. The global dynamics of the two-species case of \eqref{a.1} was recently completely classified for the weak competition regime $a_{11}a_{22} > a_{12}a_{21}$ in \cite{hn2016i,hn2016,hn2017}, assuming  $d_i$ and $a_{ij}$ are constants and $m_i(x)$ are spatially heterogenous. It was shown that  there is always a globally asymptotically stable non-negative equilibrium for the problem, and the dynamics can be completely determined according to the competition strength $a_{ij}$, the diffusion coefficients $d_i$ and heterogeneous resource functions $m_i(x)$ by using linear stability analysis and monotone dynamical system theory.


In this paper we show the global stability of the positive nonconstant equilibrium of \eqref{a.1} with spatially nonhomogeneous  $d_i$, $m_i$ and $a_{ij}$ by using a new Lyapunov functional method. For the two-species case, our result is
\begin{theorem}\label{th1.2} Suppose that $k=2$, and the functions $d_i, m_i,a_{ij}$ satisfy
 \bes\label{a.2}
m_i,a_{ij},d_i\in C^\alpha(\ol\Omega)\ {\rm and} \ d_i(x), m_i(x),a_{ij}(x)>0\ \  {\rm on}\ \ \ol\Omega.
 \ees
Assume that the problem \eqref{a.1} has a positive equilibrium solution $(u_1^*(x),u_2^*(x))$ and there exist $\beta_1>0$, $\beta_2>0$, $\td\beta_1>0$ and  $\td\beta_2>0$
such that
\bes\label{a.3}
\beta_1\leq \frac{u_1^*(x)d_2(x)}{u_2^*(x)d_1(x)}\leq \beta_2,\ \ \td\beta_1\leq \frac{u_1^*(x)}{u_2^*(x)}\leq \td\beta_2,\ \ x\in\ol\Omega.
\ees
Suppose that one of the following conditions holds:
\begin{enumerate}[leftmargin=4em]
\item[$\mathbf{(A_1)}$]\, $\dd\frac{\min_{\ol\Omega}[a_{11}(x)a_{22}(x)]}{\max_{\ol\Omega}a_{12}(x)
 \max_{\ol\Omega}a_{21}(x)}>\sqrt{\frac{\beta_2}{\beta_1}}$;
\item[$\mathbf{\bf(A_2)}$]\, $\dd\min_{\ol\Omega}\frac{a_{11}(x)a_{22}(x)}{a_{12}(x)a_{21}(x)}>\sqrt{\frac{\beta_2}{\beta_1}}$
 when $a_{12}(x)=\lambda a_{21}(x)$ for some constant $\lambda> 0$;
\item[$\mathbf{\bf(A_3)}$]\, $\dd\frac{\min_{\ol\Omega}[{a_{11}(x)a_{22}(x)}/(d_1(x)d_2(x))]}
{\max_{\ol\Omega}[a_{12}(x)/d_1(x)]\max_{\ol\Omega}[a_{21}(x)/d_2(x)]}
>\sqrt{\frac{\td\beta_2}{\td\beta_1}}$; or
 \item[$\mathbf{\bf(A_4)}$]\, $\dd\min_{\ol\Omega}\frac{a_{11}(x)a_{22}(x)}{a_{12}(x)a_{21}(x)}
>\sqrt{\frac{\td\beta_2}{\td\beta_1}}$ when $\dd \frac{a_{12}(x)}{d_1(x)}=\lambda\frac{a_{21}(x)}{d_2(x)}$ for some constant $\lambda>0$.
\end{enumerate}
Then for $i=1,2$, $\dd\lim_{t\to\yy}u_i(x,t)=u_i^*(x)$  in $C^2(\ol\Omega)$, where $(u_1,u_2)$ is the solution of \eqref{a.1} with any initial condition $(\varphi_1,\varphi_2)$ such that $\varphi_i(x)\geq, \, \not\equiv 0$.
  \end{theorem}

The assumption of existence of a positive equilibrium of \eqref{a.1} in Theorem \ref{th1.2} is not restrictive as the instability of both semitrivial equilibria implies the existence of positive coexistence state from the monotone dynamical system \cite{hess1991,Smith}. On the other hand, our result relies on an \textit{a priori}  estimate \eqref{a.3} of the positive equilibrium solutions. We will show that such estimate can be obtained by using the upper and lower solutions method in some situations (see Corollary \ref{coro3.7}). Our method can also be used to show the global stability of the positive equilibrium of \eqref{a.1} in the degenerate two-species case when one of the diffusion coefficients is zero (see Theorem \ref{thm3.1}).


Next we state the global stability of the nonnegative nonconstant equilibrium solution of \eqref{a.1} with $k\geq 3$ species in a nonhomogeneous environment. Here for convenience we only assume the resource $m_i(x)$ are spatial-dependent,  and all of $a_{ij}$ and $d_i$ are constants. We introduce some notations. Setting $1\leq i_0\leq k$ and
\bes\label{AB}
A=(a_{ij})_{k\times k},\  B=A-I_k,\ A_{0}=(a_{ij})_{i_0\times i_0}\  {\rm and} \ B_{i_0}=A_{i_0}-I_{i_0},
\ees
where $a_{ii}=1$ for $1\leq i\leq k$ and $I_k$, $I_{i_0}$ are the $k\times k$ and $i_0\times i_0$ identity matrices, respectively.
Then we have the following  of the nonnegative nonconstant equilibrium solution of \eqref{a.1}.

\begin{theorem}\label{th1.4}\, Suppose that $k\geq 2$, the functions $d_i, m_i,a_{ij}$ satisfy \eqref{a.2}, and $d_i, a_{ij}$ are constants. Assume that $I_k$, $B$, $I_{i_0}$ and $B_{i_0}$ are defined as in \eqref{AB}.

\begin{enumerate}

\item[{\rm (i)}] If there exists a $k\times k$ diagonal matrix  $Q_1$ with positive constant entries such that
  \bess
  Q_1(I_k-B-\mathbf{c}_1)+(I_k-B-\mathbf{c}_1)^T Q_1\ \ {\rm is\ positive\ definite},
  \eess
  where  the diagonal matrix $\mathbf{c}_1$ will be given by the assumption  $\mathbf{(F_{3})}$  in \S\ref{section4.1}.  Then \eqref{a.1} has a unique positive equilibrium solution $(u_1^*(x),\cdots,u_k^*(x))$, and  $\dd\lim_{t\to\yy}u_i(x,t)=u_i^*(x)$ in $C^2(\ol\Omega)$ for $1\leq i\leq k$.

\item[{\rm (ii)}] If there exist   $k\times k$ and $i_0\times i_0$ diagonal matrices    $Q_2$ and $Q_3$  with positive constant entries, respectively,  such that
  \begin{align}
  &\dd \max_{\ol\Omega}m_i-\sum_{i=1}^{i_0}a_{ij}\underline c_i<0,\ \ \dd\min_{\ol\Omega}m_i-\sum_{i=1}^{i_0}a_{ij}\bar c_i<0,\ \ \forall\;i_0+1\leq i\leq k,\label{1.5}
  \end{align}
and both $Q_2(I_k-B)+(I_k-B)^T Q_2$ and $Q_3(I_{i_0}-B_{i_0}-\mathbf{c}_2)+(I_{i_0}-B_{i_0}-\mathbf{c}_2)^T Q_3$ are positive  definite, where the positive constants $\bar c_i$, $\underline c_i$ for $1\leq i\leq i_0$ and the  diagonal matrix $\mathbf{c}_2$ will be given by the assumption  $\mathbf{(G_2)}$ in \S\ref{s4.2}. Then \eqref{a.1} has a  semitrivial equilibrium solution $(u_1^*(x),,\cdots,u_{i_0}^*(x),0,\cdots,0)$ for which $u_i^*(x)>0$ on $\ol\Omega$, and
  \begin{align*}
 &\dd\lim_{t\to\yy}u_i(x,t)=u_i^*(x)&&{\rm in}\ C^2(\ol\Omega),\ 1\leq i\leq i_0,\\
 &\dd\lim_{t\to\yy}u_i(x,t)=0&&{\rm in}\ C^2(\ol\Omega),\  i_0+1\leq i\leq k.
  \end{align*}
   \end{enumerate}
\end{theorem}
Here $\mathbf{c}_1$ and $\mathbf{c}_2$ depend on the resource density functions $m_i(x)$ for $1\leq i\leq k$, and $\mathbf{c}_1=\mathbf{c}_2=0$ if all of $m_i$ are constants. The condition \eqref{1.5} implies that the species $u_i$ for $i_0+1\leq i\leq k$ will be extinct with  barren resources. The proof of Theorem \ref{th1.4} follows from the ones of Theorems \ref{th4.6} and \ref{th4.11}, and the proof is based on a new Lyapunov functional method, upper-lower solution method and results in \cite{gbs1977}.  We also remark that the global stability results in Theorem \ref{th1.2} and Theorem \ref{th1.4} also hold if the diffusion terms  $\Delta u_i$  are replaced by a divergence form ${\rm div}(a_i(x)\nabla u_i)$  where $a_i(x)\in C^{1+\alpha}(\ol\Omega)$ with $1\leq i\leq k$,  $0<\alpha<1$ are positive functions.

The role of spatial heterogeneity in diffusive two-species competition system \eqref{a.1} have been explored in  many work, see for example \cite{cc1991,cc1993,hn2013i,hn2013ii,hlm2002,hsu1993,ln2012,lou2006} and the references therein, in which  various methods and mathematic tools  have been applied to analysis the existence and stability of the equilibrium solutions. The additional effect of advection on the diffusive two-species competition models have been considered in \cite{lzz2019,zx2018jfa} and the references therein, and the effect of nonlocal competition has been studied in \cite{nsw2018}. Note that the diffusive two-species Lotka-Volterra competition model \eqref{a.1} generates a monotone dynamical system, so the powerful tools from monotone dynamical system theory can be applied \cite{hess1991,Smith}. However, when $k\geq 3$, the monotone dynamical system theory cannot be applied to problem \eqref{a.1}. Our approach here does not rely on the monotone dynamical system methods, and the global stability proved in Theorem \ref{th1.4} for competition models with arbitrary number of species is perhaps the first such result for spatially heterogeneous  models.

A key ingredient of our work here is a new Lyapunov functional method. In \cite{gbs1977,hsu2005}, the global stability of positive equilibrium solution of \eqref{a.1}  for homogeneous environment is proved using Lyapunov functional methods when $d_i, m_i,a_{ij}$ are all constants. The Lyapunov functional there is constructed as $F_1(t)=\dd\int_{\Om} V(u_i(x,t)) {\rm d}x$, where $V(u_i)$ is the Lyapunov function for the ordinary differential equation model, and the equilibrium solution is a constant one. The integral form of the Lyapunov functional can be viewed as an unweighted average of the ODE Lyapunov function on the spatial domain. However this simple construction does not work for the spatially heterogeneous situation, and the equilibrium solution in that case is a non-constant one. In this work, we use a new Lyapunov functional in form of $F_2(t)=\dd\int_{\Om} w_*(x) V(u_i(x,t)) {\rm d}x$, which is a weighted average of the ODE Lyapunov function on the spatial domain, and the weight function $w_*(x)$ depends on the nonhomogeneous functions $d_i, m_i,a_{ij}$ and non-constant equilibrium solution (assuming it exists). Such construction is motivated by the method used in \cite{ls2010jde} for the the global stability of equilibrium solutions of  coupled ordinary differential equation models on networks (which is patchy environment or discrete spatial domain). Such Lyapunov function has also been used in \cite{kw2017} for a diffusive SIR epidemic model. To demonstrate this new method, we first prove the global stability of a non-constant equilibrium solution for a spatially heterogenous diffusive logistic model (see Theorem \ref{coro2.4}). That result is well-known but we give a new proof for the spatial heterogenous case.


This paper is organized as follows. In Section 2, we give some preliminaries, and prove the global stability of a non-constant equilibrium solution for a spatially heterogenous diffusive logistic model. In Section 3, we prove the  the global stability of  the positive equilibrium solution for the  two-species  case of \eqref{a.1}, and in  Section 4 we consider the global stability of  the non-negative equilibrium solution for \eqref{a.1} with arbitrary number of species.

\section {Preliminaries}
{\setlength\arraycolsep{2pt}

When using the Lyapunov functional method to investigate the global stability of equilibrium of reaction-diffusion systems, the uniform estimates of solutions of parabolic equations play an important role.  We first recall the following results on the uniform estimates for the second order parabolic equations. 
Consider the initial-boundary value problem
\bes\lf\{\begin{array}{lll}\label{2.1}
u_t+\mathcal{L}u=f(x,t,u),\ \  & x\in \Omega,&t>0,\\[1mm]
B[u]=0,& x\in \partial\Omega,& t>0,\\[1mm]
u(x,0)=\varphi(x), & x\in \Omega,
 \end{array}\rr.
\ees
where the domain $\Omega\subset\mathbb{R}^n$ is bounded with a smooth boundary $\partial\Omega\in C^{2+\alpha}$, operators $\mathcal{L}$ and $B$ have the forms:
\begin{equation*}
\begin{split}
&\mathcal{L}[u]=-a_{ij}(x,t)D_{ij}u+b_j(x,t)D_j u+c(x,t)u, \\
&B[u]=u,\ {\rm or} \
B[u]=\frac{\partial u}{\partial\nu}+b(x)u,
\end{split}
\end{equation*}
 with $b\in C^{1+\alpha}(\partial\Omega)$,
$0<\alpha<1$  and $b\geq 0$. The initial condition  $\varphi\in W_p^2(\Omega)$,  $p>1+n/2$,  satisfies $B[\varphi]\big|_{\partial\Omega}=0$.

Denote $Q_{\yy}=\Omega\times[0,\yy)$. We make the following assumptions:
 \begin{enumerate}[leftmargin=4em]
 \item[{\rm ($\mathbf{L_1}$)}] $a_{ij},\, b_j,\ c\in C(\ol\Omega\times [0,\yy))$ and there are positive constants $\lambda$ and $\Lambda$ such that
  \bess
  \lambda |y|^2\leq \sum_{1\leq i,j\leq n}a_{ij}(x,t)y_iy_j\leq\Lambda |y|^2,\ \ |b_j(x,t)|,\ |c(x,t)|\leq \Lambda
  \eess
for all $(x,t)\in Q_{\yy},\ y\in \mathbb{R}^n$.
\item [{\rm ($\mathbf{L_2}$)}] For any fixed $m>0$, there exists a positive constant $C(m)$ such that, for all $k\ge 1$,
  \bess
  ||a_{ij}||_{C^{\alpha,\alpha}(\ol\Omega\times[k,k+m])},\  ||b_j||_{C^{\alpha,\alpha}(\ol\Omega\times[k,k+m])},\  ||c||_{C^{\alpha,\alpha}(\ol\Omega\times[k,k+m])}\leq C(m),\ \ \forall\ x\in\ol\Omega.
  \eess
\item[{\rm ($\mathbf{L_3}$)}] $f(x,0,0)=0$ on $\partial \Omega$ when
$B[u]=u$, $f\in L^{\yy}(Q_\yy\times [\sigma_1,\sigma_2])$ for some $\sigma_1<\sigma_2$ and there exists $C(\sigma_1,\sigma_2)>0$ such that
\bess
|f(x,t,u)-f(x,t,v)|\leq C(\sigma_1,\sigma_2)|u-v|,\ \ \forall\ (x,t)\in Q_{\yy},\ u,v\in [\sigma_1,\sigma_2]
\eess
 and $f(\cdot,u)\in C^{\alpha,\alpha/2}(\ol\Omega\times[h,h+3])$ uniformly for $u\in[\sigma_1,\sigma_2]$ and $h\geq 0$, i.e., there exists a constant $C>0$ so that
\bess
|f(x,t,u)-f(y,s,u)|\leq C(|x-y|^\alpha+|t-s|^{\alpha/2})
\eess
for all $(x,t),\ (y,s)\in\ol\Omega\times[h,h+3]$, $u\in[\sigma_1,\sigma_2]$ and $h\geq 0$.
\end{enumerate}

Under the above assumptions, we have the following boundedness result for a globally defined solution $u(x,t)$ of \eqref{2.1}.
\begin{theorem}\label{th2.1}
Let $u(x,t)$ be a solution of \eqref{2.1} and $\sigma_1<u<\sigma_2$ for some $\sigma_1,\,\sigma_2\in \mathbb{R}$. Assume that $f$ satisfies {\rm ($\mathbf{L_3}$)} for these $\sigma_1,\,\sigma_2$. If $a_{ij},\, b_j$ and  $c$ satisfy the assumption {\rm ($\mathbf{L_1}$)}, then, for any given $m>0$, there is a constant $C_1(m)>0$ such that
\bess
\|u\|_{W_p^{2,1}(\ol\Omega\times [\tau,\tau+m])}\leq C_1(m),\ \ \ \forall\ \tau\ge 1.
\eess
In addition, if $a_{ij},\, b_j$ and  $c$ satisfy the assumption {\rm ($\mathbf{L_2}$)}, then, for any given $\tau\ge 1$, there is a constant $C_2(\tau)>0$ such that
\bess
\max_{x\in\ol\Omega}\|u_t(x,\cdot)\|_{ C^{\alpha/2}([\tau,\yy))}+\max_{t\geq\tau}\|u_t(\cdot,t)\|_{C(\ol\Omega)}
+\max_{t\geq\tau}\|u(\cdot,t)\|_{C^{2+\alpha}(\ol\Omega)}\leq C_2(\tau).
\eess
\end{theorem}

For the idea of  proof to Theorem \ref{th2.1}, the interested readers can  refer to the proofs of \cite[Theorem 2.1]{wmx2016} and \cite[Theorem 2.2]{WZhang18} for the details. We also recall the following calculus lemma which will be used to prove  the global stability of equilibrium solution.

\begin{lemma}{\rm(\hspace{-.1mm}\cite[Theorem 1.1]{wmx2018})}\label{th2.2} Let $\delta,\ c > 0$ be positive constants, $\psi(t)\geq0$ in $[0,\yy)$ and $\int_{\delta}^{\yy} h(t)dt<\yy$. Assume that $\varphi\in C^1([0,\yy))$ is bounded from below and satisfies
\bess
\varphi'(t)\leq-c\psi(t)+h(t)\ \ \ \text{in}\, \ [\delta,\yy).
\eess
If either $\psi\in C^1([\delta,\yy))$  and $\psi'(t)\leq K$ in $[\delta,\yy)$ for some constant $K >0$, or $\psi\in C^{\gamma}([\delta,\yy))$ and
$\|\psi\|_{C^{\gamma}([\delta,\yy))}\leq K$ for some constants $0<\gamma <1$ and $K > 0$, then $\dd\lim_{t\to\yy}\psi(t)=0$.
\end{lemma}

In order to use the Lyapunov functional method to study the global asymptotic stability of spatially nonhomogeneous positive equilibrium solutions, we should give some basic integral estimates. Given two functions $\Phi(x,u)$ and $g(x,u)$, we make the following assumptions:
 \begin{enumerate}[leftmargin=4em]
\item[$\mathbf{(H_1)}$] $\Phi \in C^{2,2}(\ol\Omega\times [0,\yy))$, $\Phi(x,0)=0$ and $\Phi_u(x,u)>0$ for $x\in\ol\Omega$ and $u>0$.
\item[$\mathbf{(H_2)}$] $g\in C^{0,1}(\partial\Omega\times [0,\yy))$, and for any $x\in\partial\Omega$, the function $\dd\frac{g(x,u)}{\Phi(x,u)}$  is nonincreasing for $u\in [0,\yy)$.
\end{enumerate}
Then we have the following integral estimates.
\begin{lemma}\label{lemma2.3}
Let $w,\,w_*\in C^{2}(\ol\Omega)$ be two positive functions and let $a\in C^1(\ol\Omega)$, $c\in C^2(\ol\Omega)$ be nonnegative functions.
\begin{enumerate}
\item[{\rm (i)}] Suppose that $\beta\geq 1$ is a constant, $\Phi(x,\tau)=c(x)\tau$, and the functions  $g$ and $\Phi$ satisfy $\mathbf{(H_2)}$.  If $\dd\f{\partial (c(x)w)}{\partial\nu}=g(x,w)$ and $\dd\f{\partial (c(x)w_*)}{\partial\nu}=g(x,w_*)$ on $\partial\Omega$, then
\bes\label{2.6}
 &&\int_{\Omega}\frac{c(x)w_*[w^\beta-w_*^\beta]}{w^\beta}\bigg({\rm div}\{a(x)\nabla [c(x)w]\}-\frac{w}{w_*}{\rm div} \{a(x)\nabla [c(x)w_*]\}\bigg){\rm d}x\nonumber\\[.5mm]
 &\leq&-\int_{\Omega} \beta ac^2w^2\bigg(\frac{w_*}{w}\bigg)^{\beta-1}\Big|\nabla \frac{w_*}{w}\Big|^2{\rm d}x\leq0.
 \ees

\item[{\rm (ii)}] Suppose that $g$ and $\Phi$ satisfy $\mathbf{(H_1)}$ and $\mathbf{(H_2)}$.  If  $\dd\f{\partial\Phi(x,w)}{\partial\nu}=g(x,w)$ and $\dd\f{\partial\Phi(x,w_*)}{\partial\nu}=g(x,w_*)$ on $\partial\Omega$, then
\begin{align}\label{2.5}
 &\int_{\Omega}\frac{\Phi(x,w_*)[\Phi(x,w)-\Phi(x,w_*)]}{\Phi(x,w)}\bigg({\rm div}[a(x)\nabla \Phi(x,w)]-\frac{\Phi(x,w)}{\Phi(x,w_*)}{\rm div} [a(x)\nabla \Phi(x,w_*)]\bigg){\rm d}x \nonumber\\[1.5mm]
 \leq&-\int_{\Omega}  a(x)[\Phi(x,w)]^2\Big|\nabla \frac{\Phi(x,w_*)}{\Phi(x,w)}\Big|^2{\rd}x\leq 0.
\end{align}
\end{enumerate}
\end{lemma}

\begin{proof} (i) It follows from Green's Theorem that
\begin{equation*}
\begin{split}
&\int_{\Omega}\frac{c(x)w_*[w^\beta-w_*^\beta]}{w^\beta}\bigg({\rm div}\{a(x)\nabla [c(x)w]\}-\frac{w}{w_*}{\rm div} \{a(x)\nabla [c(x)w_*]\}\bigg){\rm d}x\\[.5mm]
=&\int_{\Omega}\frac{w^\beta-w_*^\beta}{w^\beta}\big[cw_*{\rm div}(a\nabla (cw))-cw\,{\rm div}(a\nabla (cw_*))\big]{\rd}x\\[.5mm]
=&\int_{\Omega}\frac{w^\beta-w_*^\beta}{w^\beta}\text{div}(acw_*\nabla (cw)-acw\nabla (cw_*)){\rd}x\\[.5mm]
=&\int_{\partial\Omega} \frac{w^\beta-w_*^\beta}{w^\beta}\bigg(acw_*\frac{\partial (cw)}{\partial \nu}-acw\frac{\partial (cw_*)}{\partial \nu}\bigg){\rd}S\\[.5mm]
 &-\int_{\Omega}\bigg(\nabla\frac{w^\beta-w_*^\beta}{w^\beta}\bigg)a(cw_*\nabla (cw)-cw\nabla (cw_*)) {\rd}x\\[.5mm]
 =&\int_{\partial\Omega} \frac{w^\beta-w_*^\beta}{w^\beta}acww_*\bigg(\frac{g(x,w)}{w}
 -\frac{g(x,w_*)}{w_*}\bigg){\rd}S
 -\int_{\Omega}ac^2w^2\nabla\frac{w_*^\beta}{w^\beta}\cdot\nabla\frac{cw_*}{cw}{\rd}x\\[.5mm]
\leq&-\int_{\Omega} \beta ac^2w^2\bigg(\frac{w_*}{w}\bigg)^{\beta-1}\Big|\nabla \frac{w_*}{w}\Big|^2{\rd}x\leq 0.
\end{split}
\end{equation*}

(ii) Similar to the above computation we have
\begin{equation*}
\begin{split}
&\int_{\Omega}\frac{\Phi(x,w_*)[\Phi(x,w)-\Phi(x,w_*)]}{\Phi(x,w)}\bigg({\rm div}[a\nabla \Phi(x,w)]-\frac{\Phi(x,w)}{\Phi(x,w_*)}{\rm div} [a\nabla \Phi(x,w_*)]\bigg){\rm d}x\\[.5mm]
=&\int_{\Omega}\frac{\Phi(x,w)-\Phi(x,w_*)}{\Phi(x,w)}\text{div}[a\Phi(x,w_*)\nabla \Phi(x,w)-a\Phi(x,w)\nabla \Phi(x,w_*)]{\rd}x\\[.5mm]
=&\int_{\partial\Omega} \frac{\Phi(x,w)-\Phi(x,w_*)}{\Phi(x,w)}\bigg(a\Phi(x,w_*)\frac{\partial \Phi(x,w)}{\partial \nu}-a\Phi(x,w)\frac{\partial \Phi(x,w_*)}{\partial \nu}\bigg){\rd}S\\[.5mm]
 &-\int_{\Omega}\bigg(\nabla\frac{\Phi(x,w)-\Phi(x,w_*)}{\Phi(x,w)}\bigg)a[\Phi(x,w_*)\nabla \Phi(x,w)-\Phi(x,w)\nabla \Phi(x,w_*)]{\rd}x\\[.5mm]
 =&\int_{\partial\Omega} \frac{\Phi(x,w)-\Phi(x,w_*)}{\Phi(x,w)}a\Phi(x,w)\Phi(x,w_*)\bigg(\frac{ g(x,w)}{\Phi(x,w)}-\frac{ g(x,w_*)}{\Phi(x,w_*)}\bigg){\rd}S\\[.5mm]
 &-\int_{\Omega}a[\Phi(x,w)]^2\Big|\nabla \frac{\Phi(x,w_*)}{\Phi(x,w)}\Big|^2{\rd}x
 \leq  -\int_{\Omega}a[\Phi(x,w)]^2\Big|\nabla \frac{\Phi(x,w_*)}{\Phi(x,w)}\Big|^2{\rd}x \leq 0.
\end{split}
\end{equation*}
The proof is finished.
\end{proof}

Next, making use of the results in Lemma \ref{lemma2.3}, we study the following  scalar parabolic equation
\bes\lf\{\begin{array}{lll}\label{2.8}
	u_t=d(x)\Delta u+u[m(x)-\varphi(x)u],\ \  & x\in \Omega,& t>0,\\[1.5mm]
	\dd \partial_\nu u=0,& x\in \partial\Omega,&t>0,\\[1.5mm]
	u(x,0)=u_0(x)\geq\not\equiv 0, & x\in \Omega,
\end{array}\rr.
\ees
where $d$, $m$ and $\varphi$ satisfy
\bes\label{2.6aa}
d, m,\varphi\in C^\alpha(\ol\Omega), \ d(x)>0,\  \int_{\Omega} \frac{m(x)}{d(x)}dx\geq 0,\  m(x)\not\equiv 0\ {\rm and}\ \varphi(x)>0,\ \ \ x\in\ol \Omega.
\ees
Let $\theta_{d,m,\varphi}$ be the unique positive solution of
\begin{equation}\label{3a}
	\begin{cases}
		d\Delta \theta+\theta[m(x)-\varphi(x)\theta]=0, & x\in \Omega,\\[1mm]
		\dd \partial_\nu \theta=0,& x\in \partial\Omega.
	\end{cases}
\end{equation}
The existence of $\theta_{d,m,\varphi}$  follows  from  \cite[Proposition 3.2]{cc2003} and \cite[Proposition 2.2]{hn2013i}, and the uniqueness of  $\theta_{d,m,\varphi}$ is a consequence of \cite[Proposition 3.3]{cc2003}.

When $d(x)\equiv 1$, the global stability of $\theta_{d,m,\varphi}$ with respect to \eqref{2.8} has been shown in \cite[Proposition 3.2]{cc2003} by using the Lyapunov functional
\bess
V(u)=\int_{\Omega} \lf(\frac{1}{2}|\nabla u|^2-F(x,u)\rr){\rm d}x,
\eess
where $F(x,u)=\dd\int_0^u r[m(x)-r\varphi(x)]{\rm d}r$, and the uniqueness of  $\theta_{d,m,\varphi}$. Here we consider a more general case that $d(x)>0$ on $\ol\Omega$, and we use a different Lyapunov functional to prove the global stability of $\theta_{d,m,\varphi}$  with respect to \eqref{2.8}.

\begin{theorem}\label{coro2.4}
Assume that $u_0(x)\geq,\not\equiv 0$.  If $d$, $m$ and $\varphi$ satisfy \eqref{2.6aa},  then the  problem \eqref{2.8} has a unique positive solution $u(x,t)$, and $\dd\lim_{t\to\yy}u(x,t)=\theta_{d,m,\varphi}(x)$ in  $C^{2}(\ol\Omega)$.
\end{theorem}
\begin{proof}  Denote $$M=\max\left\{\frac{\max_{\bar\Om}m(x)}{\max_{\bar\Om} \varphi(x)},\max_{\bar\Om} u_0(x)\right\}.$$  Then $(0,M)$ is a pair of ordered  upper and lower solutions of problem \eqref{2.8}. This implies that the  problem \eqref{2.8} has a unique positive solution $u(x,t)$ satisfying $0<u(x,t)\leq M$ for $(x,t)\in\ol\Omega\times(0,\yy)$. It follows from Theorem \ref{th2.1} that  there exists a constant $C>0$ such that
\bes\label{2.5a}
\max_{t\geq1}\|u_t(\cdot,t)\|_{C(\ol\Omega)}+\max_{t\geq1}\|u(\cdot,t)\|_{C^{2+\alpha}(\ol\Omega)}\leq C.
\ees
		
For simplicity of notations, we denote $\theta=\theta_{d,m,\varphi}$ and $f(x,u)=u[m(x)-\varphi(x)u]$. Define a function $Q: [0,\infty)\to \mathbb{R}$ by
\bess
Q(t)=\int_{\Om}\int_{\theta(x)}^{u(x,t)}\frac{\theta(x)}{d(x)}\times\frac{s-\theta(x)}{s} \text{d}s\text{d}x.
\eess
Then $Q(t)\geq 0$ for $t\geq 0$. From \eqref{2.6},  we obtain
\bes
\frac{dQ(t)}{dt}&=&\int_{\Omega}\frac{\theta(u-\theta)}{du}u_t\text{d}x
			=\int_{\Omega}\frac{\theta(u-\theta)}{du}\lf[d\Delta u+f(x,u)\rr]\text{d}x\nonumber\\[1.5mm]
			&=&\int_{\Omega}\frac{\theta(u-\theta)}{du}\lf(d\Delta u+f(x,u)
			-\frac{u}{\theta}d\Delta \theta-\frac{u}{\theta}f(x,\theta)\rr)\text{d}x\nonumber\\[1.5mm]
			&=&\int_{\Omega}\frac{\theta(u-\theta)}{du}\lf(d\Delta u-\frac{u}{\theta}
			d\Delta \theta\rr)\text{d}x +\int_{\Omega}\frac{\theta(u-\theta)}{d}\lf[\frac{f(x,u)}{u}-\frac{f(x,\theta)}{\theta}\rr]\text{d}x\label{2.8b1}\\[1.5mm]
			&\leq & \int_{\Omega} \lf(-u^2\lf|\nabla \frac{\theta}{u}\rr|^2+\frac{\theta(u-\theta)}{d}\lf[\frac{f(x,u)}{u}-\frac{f(x,\theta)}{\theta}\rr]\rr)\text{d}x\nonumber\\[1.5mm]
			&\leq& -\int_{\Omega} \frac{\theta(u-\theta)}{d}\lf[\frac{f(x,u)}{u}-\frac{f(x,\theta)}{\theta}\rr]\text{d}x=-\int_{\Omega} \frac{\theta\varphi(u-\theta)^2}{d}\text{d}x\nonumber\\[1.5mm]
			&=:&\psi(t)\leq 0\nonumber.
	\ees
	Taking advantages of \eqref{2.5a}, we have $|\psi'(t)|<C_1$ in $[1,\yy)$ for some $C_1>0$. Then it following from Lemma  \ref{th2.2} that
	\bes\label{2.6a}
	\lim_{t\to\yy}\psi(t)=-\lim_{t\to\yy}\int_{\Omega} \frac{\theta\varphi(u-\theta)^2}{d}\text{d}x=0.
	\ees
	The estimate \eqref{2.5a} also implies that the set $\{u(\cdot,t):t\geq 1\}$ is relatively compact
	in $C^2(\ol\Omega)$. Therefore, we may assume that
	\[\|u(x,t_k)-u_{\yy}(x)\|_{C^2(\ol\Omega)}\to 0\ \ \ {\rm as}\ \ t_k\to\yy\]
	for some function $u_{\yy}\in C^2(\ol\Omega)$. Combining this with \eqref{2.6a}, we can conclude that $u_{\yy}(x)\equiv \theta(x)$ for $x\in\ol\Omega$. Thus $\dd\lim_{t\to\yy}u(x,t)=\theta(x)$ in $C^2(\ol\Omega)$.
	The proof is finished.
\end{proof}

\begin{remark}
For the  quasilinear parabolic problem with nonlinear diffusion and nonlinear boundary condition:
\begin{eqnarray}\label{fujia1}
\lf\{\begin{array}{lll}
	u_t=d(x)\rm{div} [a(x)\nabla \Phi(x,u)]+f(x,u),\ \  & x\in \Omega,& t>0,\\[1.5mm]
	\dd \f{\partial \Phi(x,u)}{\partial\nu}=g(x,u),& x\in \partial\Omega,&t>0,\\[1.5mm]
	u(x,0)=u_0(x)\geq\not\equiv 0, & x\in \Omega,
\end{array}\rr.
\end{eqnarray}
where $g$ and $\Phi$ satisfy $\mathbf{(H_1)}$ and $\mathbf{(H_2)}$, $a\in C^{1+\alpha}(\ol\Omega)$,  $d\in C^\af(\ol\Omega)$ with $0<\alpha<1$ and $a(x)>0$,  $d(x)>0$ on $\ol\Omega$, one may construct a similar Lyapunov functional to prove the uniqueness and global stability of the positive equilibrium solution  $u_*$ with respect to \eqref{fujia1},
\bess
F(t)=\int_{\Om}\int_{u_*(x)}^{u(x,t)} \frac{\Phi(x,u_*)}{d(x)}\frac{\Phi(x,s)-\Phi(x,u_*)}{\Phi(x,s)} \text{d}s\text{d}x.
\eess
For more results about the problem \eqref{fujia1},   readers can  refer to {\rm\cite{mc2006cana,pao2007na}} and the references therein.
\end{remark}



\section {Two species competition model}\label{section3}
In this section, we consider the global stability of positive equilibrium of the following two species Lotka-Volterra competition model in heterogenous environment:
\bes\lf\{\begin{array}{lll}\label{3a1}
\partial_t u_1=d_1(x)\Delta u_1+u_1[m_1(x)-a_{11}(x)u_1-a_{12}(x)u_2],\ \ & x\in \Omega,& t>0,\\[1.5mm]
\partial_t u_2=d_2(x)\Delta u_2+u_2[m_2(x)-a_{21}(x)u_1-a_{22}(x)u_2], & x\in \Omega,& t>0,\\[1.5mm]
\dd\partial_\nu u_1=\partial_\nu u_2=0,& x\in \partial\Omega,&
 t>0,\\[1.5mm]
u_1(x,0)=\varphi_1(x),\ u_2(x,0)=\varphi_2(x), & x\in \Omega,
 \end{array}\rr.
\ees
where the functions $m_i,a_{ij},d_i$ satisfy \eqref{a.2}.

We first consider the degenerate case of  \eqref{3a1} with an immobile species which has zero diffusion coefficient.
\begin{theorem}\label{thm3.1}  Assume that the initial functions $\phi_i\in C(\ol\Omega)$ $(i=1,2)$ satisfy $\phi_1(x)\geq,\not\equiv 0$ and $\phi_2(x)> 0$ on $\ol \Omega$. Let $d_2(x)\equiv 0$ for $x\in \bar{\Om}$,  the functions $d_1, m_i,a_{ij}$ satisfy  \eqref{a.2}  and
	\bes\label{3.2b1}
	a_{12}(x)a_{21}(x)<a_{11}(x)a_{22}(x), \ \ \ \ x\in\bar{\Omega}.
	\ees
\begin{enumerate}
\item[{\rm (i)}]	If
\bes\label{3.3b2}
\int_{\Omega} \frac{1}{d_1(x)a_{22}(x)}\lf[a_{22}(x)m_1(x)-{a_{12}(x)}m_2(x)\rr]\rd x>0,
\ees
and
	\bes\label{3.3b1}
\min_{\bar{\Om}} \frac{m_2(x)}{a_{21}(x)}>\max_{\bar{\Om}}\frac {a_{22}(x)m_1(x)-{a_{12}(x)}m_2(x)}{a_{11}(x)a_{22}(x)-a_{12}(x)a_{21}(x)},
	\ees
	then the problem \eqref{3a1} has a  positive equilibrium solution $(u_1^*(x),u_2^*(x))$,  and  $\dd\lim_{t\to\yy}u_1(x,t)=u_1^*(x)$  in $C^1(\ol\Omega)$  and $\dd\lim_{t\to\yy}u_2(x,t)=u_2^*(x)$ in $L^2(\Omega)$.
	
\item[{\rm (ii)}]	If
	\bes\label{3.4b1}
	\frac{m_2(x)}{a_{21}(x)}\leq \theta_{d_1,m_1,a_{11}}(x),\ \ \ \ x\in\bar{\Omega},
	\ees
	then $\dd\lim_{t\to\yy}u_1(x,t)=\theta_{d_1,m_1,a_{11}}(x)$  in $C^1(\ol\Omega)$  and $\dd\lim_{t\to\yy}u_2(x,t)=0$ in $L^2(\Omega)$, where $\theta_{d_1,m_1,a_{11}}$ is defined as in \eqref{3a}.
	
\item[{\rm (iii)}]	If
	\bes\label{3.5b1}
	\frac{a_{22}(x)}{a_{12}(x)} \leq \frac{m_2(x)}{m_1(x)},\ \ \ \ x\in\bar{\Omega},
	\ees
	then $\dd\lim_{t\to\yy}u_1(x,t)=0$  in $C^1(\ol\Omega)$  and $\dd\lim_{t\to\yy}u_2(x,t)=\frac{m_2(x)}{a_{22}(x)}$ in $L^2(\Omega)$.
\end{enumerate}
\end{theorem}

\begin{proof}  (i)  Let  $(u_1^*(x),u_2^*(x))$ be an  equilibrium solution of \eqref{3a1}. Then   $(u_1^*(x),u_2^*(x))$  satisfies
	\bes\lf\{\begin{array}{lll}\label{3a2}
		-d_1(x)\Delta u_1=u_1\left[m_1(x)-\dd\frac{a_{12}}{a_{22}} m_2(x)-(a_{11}-\frac{a_{12}a_{21}}{a_{22}})u_1\right],\ \ & x\in \Omega,\\[1.5mm]
		\dd\partial_\nu u_1=0,& x\in \partial\Omega,
	\end{array}\rr.
	\ees
	and $u_2^*=\dd\frac{m_2-a_{21}u_1^*}{a_{22}}$. Since \eqref{3.3b2} holds,   the problem \eqref{3a2} has a unique positive solution $u_1^*(x)$ as it is in a form of \eqref{3a} and \eqref{2.6aa} is satisfied.  And from the maximum principle of elliptic equations, it follows that $\dd u_1^*<\max_{\bar{\Om}}\dd\frac{m_1a_{22}-a_{12}m_2}{a_{11}a_{22}-a_{12}a_{21}}$. Then making using of \eqref{3.3b1}, we obtain that $u_2^*=\dd\frac{m_2-a_{21}u_1^*}{a_{22}}>0$ on $\ol\Omega$. Hence a unique positive equilibrium solution  $(u_1^*,u_2^*)$ of \eqref{3a1} exists.
	
	   Define a
	function $F: [0,\infty)\to \mathbb{R}$ by
	$$F(t)= \int_{\Om}\int_{u_1^*(x)}^{u_1(x,t)} \frac{u_1^*(x)}{d_1(x)}\times \frac{s-u_1^*(x)}{s} {\rm d}s{\rm d}x+\int_{\Om}\int_{u_2^*(x)}^{u_2(x,t)} \xi(x)\frac{s-u_2^*(x)}{s} {\rm d}s{\rm d}x,$$
	where $\xi(x)>0$ on $\bar \Omega$ will be specified latter. Then $F(t)\ge 0$. From \eqref{2.6}, \eqref{2.8b1} and  \eqref{3a1}, we obtain
	\begin{align}\label{3.7b1}
	F'(t)=& \int_{\Omega}\frac{u_1^*(u_1-u_1^*)}{d_1u_1}\partial_t u_1{\rm d}x+ \int_{\Omega}\frac{\xi(x)(u_2-u_2^*)}{u_2}\partial_t u_2{\rm d}x\nonumber\\[1mm]
	=& \int_{\Omega}\frac{u_1^*(u_1-u_1^*)}{u_1}(\Delta u_1-\frac{u_1}{u_1^*}\Delta u_1^*){\rm d}x\nonumber\\[1mm]
	&+ \int_{\Om} \frac{u_1^*}{d_1} (u_1-u_1^*)[-a_{11}(u_1-u_1^*)-a_{12}(u_2-u_2^*)] {\rm d}x\nonumber\\[1mm]
	&+\int_{\Om} \xi(x)(u_2-u_2^*)[-a_{22}(u_2-u_2^*)-a_{21}(u_1-u_1^*)]{\rm d}x\nonumber\\[1mm]
	\leq &-\int_{\Omega} u_1^2\Big|\nabla \frac{u_1^*}{u_1}\Big|^2
	{\rm d}x-\int_{\Om}  a_{11}\frac{u_1^*}{d_1}(u_1-u_1^*)^2{\rm d}x\nonumber\\[1mm]
	&-\int_{\Om} \bigg[\lf( a_{12} \frac{u_1^*}{d_1}+\xi(x)a_{21} \rr)(u_1-u_1^*)(u_2-u_2^*)+\xi(x)a_{22} (u_2-u_2^*)^2\bigg]{\rm d}x.
	\end{align}
	Choose $0<\delta\ll 1$ and $\dd{\xi}(x)=\frac{a_{12}(x)u_1^*(x)}{d_1(x)a_{21}(x)}$. It then follows from \eqref{3.2b1} that
	\begin{align}
	&2\sqrt{\xi(a_{11}-\delta)(a_{22}-\delta)\frac{u_1^*}{d_1}}-\left( a_{12}\frac{u_1^*}{d_1}+\xi a_{21} \right)\nonumber\\[1mm]
	=&2\sqrt{\xi(a_{11}-\delta)(a_{22}-\delta)\frac{u_1^*}{d_1}}-2\sqrt{\xi a_{21}  a_{12}\frac{u_1^*}{d_1}}>0\nonumber
	\end{align}
	This combined with \eqref{3.7b1} allows us to derive
	\begin{equation*}
	\begin{split}
	F'(t)\leq &-\int_{\Om} \left[ \delta \frac{u_1^*}{d_1}(u_1-u_1^*)^2+\xi \delta (u_2-u_2^*)^2\right]{\rm d}x:=\psi(t).
	\end{split}
	\end{equation*}
	
Next we show the global stability of the positive equilibrium solution $(u_1^*,u_2^*)$. Clearly, both $u_1$ and $u_2$ are bounded in $\ol\Omega\times [0,\yy)$. Then with the help of Theorem \ref{th2.1} and Sobolev embedding theorem, there exists a constant $C>0$ such that
\bes\label{3.8b1}
\max_{t\geq1}\|u_1(\cdot,t)\|_{C^{1+\alpha}(\ol\Omega)}
\leq C \ \ \ {\rm for \ some}\ \ 0<\alpha<1.
\ees
Taking advantages of \eqref{3a1}, \eqref{3.8b1} and  the definition of $\psi(t)$, we get   $|\psi'(t)|<C_1$ in $t\in[1,\yy)$ for some $C_1>0$. Then it follows from Lemma \ref{th2.2} that $\dd\lim_{t\to\yy}\psi(t)=0$, which implies that
\bess
\lim_{t\to\yy}u_1(x,t)=u_1^*(x), \ \ \lim_{t\to\yy}u_2(x,t)=u_2^*(x) \ \ {\rm in} \ \  L^2(\ol\Omega).
\eess
Similarly to the discussion of  Theorem \ref{coro2.4}, we can prove  $\dd\lim_{t\to\yy}u_1(x,t)=u_1^*(x)$  in $C^1(\ol\Omega)$.
	
(ii)   Clearly, $(\theta_{d_1,m_1,a_{11}},0)$ is a semi-trivial equilibrium solution of \eqref{3a1}, where $\theta_{d_1,m_1,a_{11}}(x)$ is the unique positive solution of  \eqref{3a}. Define a
function $F: [0,\infty)\to \mathbb{R}$ by
$$F(t)= \int_{\Om}\int_{\theta_{d_1,m_1,a_{11}(x)}}^{u_1(x,t)} \frac{\theta_{d_1,m_1,a_{11}}(x)}{d_1(x)}\times\frac{s-\theta_{d_1,m_1,a_{11}}(x)}{s} {\rm d}s{\rm d}x+\int_{\Om}\xi(x)u_2(x,t){\rm d}x,$$
where $\xi(x)=\dd\frac{a_{12}\theta_{d_1,m_1,a_{11}(x)}}{d_1(x)a_{21}(x)}$.  Here we simply denote $\theta=\theta_{d_1,m_1,a_{11}}$. From \eqref{2.6}, \eqref{2.8b1}  and  \eqref{3.4b1}, we get
\begin{align*}
F'(t)=& \int_{\Omega}\lf(\frac{\theta(u_1-\theta)}{u_1}(\Delta u_1-\frac{u_1}{\theta}\Delta \theta){\rm d}x\nonumber+ \frac{\theta}{d_1} (u_1-\theta)[-a_{11}(u_1-\theta)-a_{12}u_2] \rr){\rm d}x\nonumber\\[1mm]
&+\int_{\Om} \xi u_2[(m_2-a_{21}\theta)-a_{21}(u_1-\theta)-a_{22}u_2]{\rm d}x\nonumber\\[1mm]
\leq &-\int_{\Om}  \bigg[a_{11}\frac{\theta}{d_1}(u_1-\theta)^2+\lf( a_{12} \frac{\theta}{d_1}+\xi a_{21} \rr)(u_1-\theta)u_2+\xi a_{22} u_2^2\bigg]{\rm d}x.
\end{align*}
Then by the same arguments as part (i), we obtain  the desired conclusion.

(iii)  Clearly, $(u_1^*(x),u_2^*(x))=\lf(0,\dd\frac{m_2(x)}{a_{22}(x)}\rr)$ is a semi-trivial equilibrium solution of \eqref{3a1}. Define a function $F: [0,\infty)\to \mathbb{R}$ by
$$F(t)= \int_{\Om}\frac{u_1}{d_1} {\rm d}x+\int_{\Om}\int_{u_2^*(x)}^{u_2(x,t)} \xi(x)\frac{s-u_2^*(x)}{s} {\rm d}s{\rm d}x,$$
where $\xi(x)=\dd\frac{a_{12}(x)}{d_1(x)a_{21}(x)}$.  From  \eqref{3.5b1}, we have
\begin{align*}
F'(t)=&\int_{\Om} \lf(\frac{u_1(x,t)}{d_1(x)}[(m_1-a_{12}u_2^*)-a_{11}u_1-a_{12}(u_2-u_2^*)]+\xi (u_2-u_2^*)[-u_1-a_{22}(u_2-u_2^*)]\rr){\rm d}x\nonumber\\[1mm]
\leq &-\int_{\Om}  \bigg[\frac{a_{11}}{d_1}u_1^2+\lf(\frac{a_{12}}{d_1}+\xi a_{21} \rr)u_1(u_2-u_2^*)+\xi a_{22} (u_2-u_2^*)^2\bigg]{\rm d}x.
\end{align*}
Then by the same arguments as part (i), we get  the desired conclusion. The proof is completed.
 \end{proof}

 Next we prove the global stability of the positive equilibrium solution $(u_1^*,u_2^*)$ using Lyapunov functional method.
\begin{proof}[Proof of Theorem {\rm\ref{th1.2}}] We first assume that $\mathbf{\bf(A_1)}$  holds. Define a
function $F: [0,\infty)\to \mathbb{R}$ by
$$F(t)= \int_{\Om}\int_{u_1^*(x)}^{u_1(x,t)} \frac{u_1^*(x)}{d_1(x)}\times \frac{s-u_1^*(x)}{s} {\rm d}s{\rm d}x+\xi \int_{\Om}\int_{u_2^*(x)}^{u_2(x,t)} \frac{u_2^*(x)}{d_2(x)}\times \frac{s-u_2^*(x)}{s} {\rm d}s{\rm d}x,$$
where the constant $\xi>0$ will be specified  latter. Then $F(t)\ge 0$. From \eqref{2.6}, \eqref{2.8b1} and  \eqref{3a1}, we obtain
 \begin{align}
 F'(t)=& \int_{\Omega}\frac{u_1^*(u_1-u_1^*)}{u_1}(\Delta u_1-\frac{u_1}{u_1^*}\Delta u_1^*){\rm d}x+ \xi\int_{\Omega}\frac{u_2^*(u_2-u_2^*)}{u_2}(\Delta u_2-\frac{u_2}{u_2^*}\Delta u_2^*){\rm d}x\nonumber\\[1mm]
&+ \int_{\Om} \frac{u_1^*}{d_1} (u_1-u_1^*)[-a_{11}(u_1-u_1^*)-a_{12}(u_2-u_2^*)] {\rm d}x\nonumber\\[1mm]
&+\xi \int_{\Om} \frac{u_2^*}{d_2} (u_2-u_2^*)[-a_{22}(u_2-u_2^*)-a_{21}(u_1-u_1^*)]{\rm d}x\nonumber\\[1mm]
\leq &-\int_{\Omega}\bigg( u_1^2\Big|\nabla \frac{u_1^*}{u_1}\Big|^2+\xi u_2^2\Big|\nabla \frac{u_2^*}{u_2}\Big|^2
\bigg){\rm d}x-\int_{\Om}  a_{11}\frac{u_1^*}{d_1}(u_1-u_1^*)^2{\rm d}x\nonumber\\[1mm]
&-\int_{\Om} \bigg[\left( a_{12} \frac{u_1^*}{d_1}+\xi a_{21} \frac{u_2^*}{d_2}\right)(u_1-u_1^*)(u_2-u_2^*)+\xi a_{22} \frac{u_2^*}{d_2}(u_2-u_2^*)^2\bigg]{\rm d}x.\label{3.9}
 \end{align}
Choose $0<\delta\ll 1$ and $\dd{\xi}{}=(a_{12}^+/a_{21}^+)\sqrt{\beta_1\beta_2}$ where $a_{12}^+=\max_{\ol\Omega}a_{12}(x)$ and $a_{21}^+=\max_{\ol\Omega}a_{21}(x)$. It then follows from \eqref{a.3} and $\mathbf{\bf(A_1)}$ that
\begin{align}
&2\sqrt{\xi(a_{11}-\delta)(a_{22}-\delta)\frac{u_1^*u_2^*}{d_1d_2}}-\left( a_{12}\frac{u_1^*}{d_1}+\xi a_{21} \frac{u_2^*}{d_2}\right)\nonumber\\[1mm]
=&2\sqrt{ \xi \frac{u_1^*u_2^*}{d_1d_2}}\sqrt{(a_{11}-\delta)(a_{22}-\delta)}-
\sqrt{ \xi \frac{u_1^*u_2^*}{d_1d_2}}\bigg(a_{12}\sqrt{\frac{ u_1^*d_2}{\xi u_2^*d_1}}+a_{21}\sqrt{\frac{\xi u_2^*d_1}{ u_1^*d_2}}\bigg)\nonumber\\[1mm]
\geq& \sqrt{ \xi \frac{u_1^*u_2^*}{d_1d_2}}\bigg[2\sqrt{(a_{11}-\delta)(a_{22}-\delta)}-
\bigg(a_{12}^+\sqrt{\frac{1}{\xi}\beta_2}+a_{21}^+\sqrt{\frac{\xi }{ \beta_1}}\bigg)\bigg]\nonumber\\[1mm]
\geq&\sqrt{ \xi \frac{u_1^*u_2^*}{d_1d_2}}\bigg(2\sqrt{ (a_{11}-\delta)(a_{22}-\delta)}-2\sqrt{a_{12}^+a_{21}^+\sqrt{\beta_2/\beta_1}}\bigg)> 0.\nonumber
\end{align}
This combined with \eqref{3.9} allows us to derive
\begin{equation*}
\begin{split}
F'(t)\leq &-\int_{\Om} \left[ \delta \frac{u_1^*}{d_1}(u_1-u_1^*)^2+\xi \delta \frac{u_2^*}{d_2}(u_2-u_2^*)^2\right]{\rm d}x:=\psi(t).
\end{split}
\end{equation*}
Making use of Theorem \ref{th2.1} and Lemma \ref{th2.2}, by the similar arguments  as  in the proof of Corollary  \ref{coro2.4},  we can prove  $\dd\lim_{t\to\yy}u_1(x,t)=u_1^*(x)$ and $\dd\lim_{t\to\yy}u_2(x,t)=u_2^*(x)$ in $C^2(\ol\Omega)$.

When one of $\mathbf{\bf(A_2)}$, $\mathbf{\bf(A_3)}$ and $\mathbf{\bf(A_4)}$ holds, the proof is similar and the details are omitted here. The proof of Theorem \ref{th1.2} is finished.
\end{proof}

To show the estimate \eqref{a.3} is achievable, we consider the following problem
  	\bes\lf\{\begin{array}{lll}\label{3a3}
	\partial_t u_1=d_1(x)\Delta u_1+u_1[\td m_1 \psi(x)+\var_1 f_1(x)-\td a_{11} \psi(x)u_1-\td a_{12}\psi(x)u_2],& x\in \Omega, t>0,\\[.5mm]
\partial_t u_2=d_2(x)\Delta u_2+u_2[\td m_2 \psi(x)+\var_2 f_2(x)-\td a_{21} \psi(x)u_1-\td a_{22}\psi(x)u_2],& x\in \Omega, t>0,\\[.5mm]
		\dd\partial_\nu u_1=\partial_\nu u_2=0,& x\in \partial\Omega,
		t>0,\\[.5mm]
	u_1(x,0)=\varphi_1(x),\ u_2(x,0)=\varphi_2(x), & x\in \Omega,
	\end{array}\rr.
	\ees
	where $i\in \{1,2\}$, $\td m_i,\,\td a_{ij},\,\var_i$ are all
	positive constants,  $\psi,\,d_i\,f_i\in C^{\alpha}(\ol\Omega)$ and
	$\psi(x),d_i(x)>0$ on $\ol\Om$. Now we apply  Theorem \ref{th1.2} to study the global stability of the positive equilibrium solutions of problem \eqref{3a3}.

\begin{corollary}\label{coro3.7} Assume that the initial functions $\varphi_i\in C(\ol\Omega)$ $(i=1,2)$ satisfy $\varphi_i(x)\geq,\not\equiv 0$. If $0\leq \var_i  \ll 1$ and
	\bess
	\frac{\td a_{21}}{\td a_{11}}<\frac{\td m_2}{\td m_1}<\frac{\td a_{22}}{\td a_{12}}, \ \ \ \frac{\td a_{11} \td a_{22}}{\td a_{12} \td a_{21}}>\sqrt{\frac{\max_{\ol\Omega}d_2/d_1}{\min_{\ol\Omega}d_2/d_1}}.
	\eess
	Then the problem \eqref{3a3} has a positive equilibrium solution  $(u_1^*(x),u_2^*(x))$ which  is globally asymptotically stable.
\end{corollary}

\begin{proof} Denote by $(u_1^*(x),u_2^*(x))$ any positive solution of the following elliptic problem
	\bes\lf\{\begin{array}{lll}\label{3.16}
		d_i(x)\Delta u_i^*+u_i^*[\td m_i \psi(x)+\var_i f_i(x)-\td a_{i1} \psi(x)u_1^*-\td a_{i2}\psi(x)u_2^*]=0,\ \ & x\in \Omega,\\[1.5mm]
		\dd\partial_\nu u_i^*=0,& x\in \partial\Omega.
	\end{array}\rr.
	\ees
	Set $\bar r_i=\max_{x\in\ol\Omega}{f_i(x)}/{\psi(x)}$, $\underline r_i=\min_{x\in\ol\Omega}{f_i(x)}/{\psi(x)}$ for $i=1,2$. Owing to ${\td a_{21}}/{\td a_{11}}<{\td m_2}/{\td m_1}<{\td a_{22}}/{\td a_{12}}$ and $0<\var,\,\eta \ll 1$, we get
	\bess
	\frac{\td a_{21}}{\td a_{11}}<\frac{\td  m_2+\var_2 \bar r_2}{\td m_1+\var_1 \underline r_1}<\frac{\td a_{22}}{\td a_{12}},\ \ \ \frac{\td a_{21}}{\td a_{11}}<\frac{\td  m_2+\var_2 \underline r_2}{\td m_1+\var_1 \bar r_1}<\frac{\td a_{22}}{\td a_{12}},
	\eess
	which implies that the linear system
	\bess
	\begin{cases}
		(\td m_1+\var_1\bar r_1) -\td a_{11} \bar u_1-\td a_{12} \underline u_2=0, \\
		(\td m_1+\var_1 \underline r_1) -\td a_{11} \underline u_1-\td a_{12}\bar u_2=0, \\
		(\td m_2+\var_2 \bar r_2 )-\td a_{21} \underline u_1-\td a_{22} \bar u_2=0,\\
		(\td m_2+\var_2 \underline r_2 )-\td a_{21} \bar u_1-\td a_{22} \underline u_2=0
	\end{cases}
	\eess
	has a unique positive solution $(\bar u_1,\underline u_1, \bar u_2,\underline u_2)$. A direct
	calculation gives
	\bess
	\underline u_1=\frac{\td a_{22} (\td m_1+\var_1 \underline r_1)-\td a_{12}(\td  m_2+\var_2 \bar r_2)}{\td a_{11} \td a_{22}-\td a_{12}\td a_{21}},\ \ \bar u_1=
	\frac{\td a_{22} (\td m_1+\var_1 \bar r_1)-\td a_{12}(\td m_2+\var_2 \underline r_2)}{\td a_{11} \td a_{22}-\td a_{12}\td a_{21}},\\[2mm]
	\underline u_2=\frac{\td a_{11} (\td m_2+\var_2 \underline r_2)-\td a_{21}(\td m_1+\var_1 \bar r_1)}{\td a_{11} \td a_{22}-\td a_{12}\td a_{21}},\ \ \bar u_2=
	\frac{\td a_{11} (\td m_2+\var_2 \bar r_2)-\td a_{21}(\td m_1+\var_1 \underline r_1)}{\td a_{11} \td a_{22}-\td a_{12}\td a_{21}}.
	\eess
	Clearly, $\bar u_1\geq \underline u_1$ and $\bar u_2\geq \underline u_2$.
	It is easily seen that $(\bar u_1, \bar u_2)$ and $(\underline u_1,\underline u_2)$ is a pair of ordered upper and lower solutions of \eqref{3.16}. Consequently  the problem \eqref{3a3} has a positive equilibrium solution $(u_1^*(x),u_2^*(x))$,  and
	\begin{equation*}
	0<\underline u_1\leq u_1^*(x)\leq \bar u_1,\;\;\; 0<\underline u_2\leq u_2^*(x)\leq\bar u_2,\ \ \  \forall\;x\in\ol\Omega.
	\end{equation*}
	Define
	\[\beta_1=\dd\frac{\underline u_1}{\bar u_2}\min_{\ol\Omega}\frac{d_2(x)}{d_1(x)},
	\ \ \ \beta_2=\frac{\bar u_1}{\underline u_2}\max_{\ol\Omega}\frac{d_2(x)}{d_1(x)}.\]
	Then
	\[\beta_1\leq\frac{u_1^*(x)d_2(x)}{u_2^*(x)d_1(x)}\leq \beta_2,\]
	and
	\bess
	&\dd\frac{\beta_2}{\beta_1}=\frac{\td a_{22} (\td m_1+\var_1 \bar r_1)-\td a_{12}(\td m_2+\var_2 \underline r_2)}{\td a_{11}(\td m_2+\var_2 \underline r_2)-\td a_{21}(\td m_1+\var_1 \bar r_1)}\times\frac{\td a_{11} (\td m_2+\var_2 \bar r_2)-\td a_{21}(\td m_1+\var_1 \underline r_1)}{\td a_{22} (\td m_1+\var_1 \underline r_1)-\td a_{12}(\td  m_2+\var_2 \bar r_2)}
	\times \frac{\max_{\ol\Omega}d_2/d_1}{\min_{\ol\Omega}d_2/d_1}.&
	\eess
	Since $0<\var,\,\eta \ll 1$ and
	\[\frac{\td a_{11} \td a_{22}}{\td a_{12} \td a_{21}}>\sqrt{\frac{\max_{\ol\Omega}d_2/d_1}{\min_{\ol\Omega}d_2/d_1}},\]
	it follows that
	\bess
	\dd\min_{x\in\ol\Omega}\frac{\td a_{11} \phi(x)\td a_{22} \phi(x)}{\td a_{12}\phi(x)\td a_{21} \phi(x)}=\frac{\td a_{11} \td a_{22}}{\td a_{12} \td a_{21}}>\sqrt{\frac{\beta_2}{\beta_1}}.
	\eess
	Thus, by Theorem \ref{th1.2}, $\dd\lim_{t\to\yy}u_1(x,t)=u_1^*(x)$ and $\dd\lim_{t\to\yy}u_2(x,t)=u_2^*(x)$ uniformly for $x\in \ol\Omega$. The proof is finished.
\end{proof}

\section{$k$ species competition models}

In this section we prove Theorem \ref{th1.4}, and  the details are contained in the following Theorems \ref{th4.6} and \ref{th4.11}.

\subsection{Global stability of positive equilibrium solution}\label{section4.1}

We consider a Lotka-Volterra competition model with $k$ species
\begin{equation}\label{4.1}
\begin{cases}
\dd \frac{\partial u_i}{\partial t}=d_i\Delta u_i+u_i\bigg(m_i(x)-\sum_{ j=1}^{ k} a_{ij}u_j\bigg), & x\in \Omega,\; \ t>0,\; 1\leq i\leq k,\\
\dd \partial_\nu u_i=0,& x\in \partial\Omega,\;
 t>0,\; 1\leq i\leq k,\\
u_i(x,0)=\varphi_i(x)\geq, \, \not\equiv 0, & x\in\Omega, \; 1\leq i\leq k,
 \end{cases}
\end{equation}
where $d_i>0$ and $a_{ij}\geq 0$ are constants. Without loss of generality, we assume that
\bes\label{aij}
a_{ij}\geq 0,\ \ a_{ii}=1.
\ees
The functions  $m_i\in C^{\alpha}(\ol\Omega)$  and satisfy  $m_i(x)>0$ on $\bar\Om$. In the following, we will investigate the global stability of $\mathbf{u}^*=(u_1^*(x),...,u_k^*(x))$  which is a positive solution of the  elliptic problem
\begin{equation}\label{4a1}
\begin{cases}
\dd d_i\Delta u_i+u_i\bigg(m_i(x)-\sum_{ j=1}^{ k} a_{ij}u_j\bigg)=0, & x\in \Omega,\ \;1\leq i\leq k,\\
\dd \partial_\nu u_i=0,& x\in \partial\Omega, \;1\leq i\leq k.
 \end{cases}
\end{equation}

If the resource functions $m_i(x)$  are positive constants,  the following result is well known.

\begin{theorem}{\rm(\hspace{-.1mm}\cite[Page 138]{gbs1977})}\label{th4.1}\, Assume that $m_i(x)$ for
$1\leq i\leq k$ are positive constants and $A=(a_{ij})_{k\times k}$. If the problem \eqref{4.1} has a positive equilibrium
 $\mathbf{u}^*$ and there exists a diagonal matrix $Q$ with positive constant entries such that $QA+A^TQ$ is positive definite. Then $\mathbf{u}^*$ is globally asymptotically stable with respect to \eqref{4.1}.
\end{theorem}
Indeed the result in \cite{gbs1977} is only for ordinary differential equation model without diffusion, but the same Lyapunov functional method can be applied to prove the global stability with respect to \eqref{4.1}.
On the other hand, if one of $m_i(x)$ is not constant, the global stability of positive equilibrium of  \eqref{4.1} cannot be obtained directly from the method for proving Theorem \ref{th4.1}.

For the simplicity of notations,  we define, for $1\leq i\leq k$,
\bes\lf\{\begin{array}{lll}\label{4.3}
\dd m^-_i=\min_{x\in\ol\Omega} m_i(x)>0,\  \ m^+_i=\max_{x\in\ol\Omega} m_i(x)>0,\\[2mm]
\mathbf{m}^-=( m^-_1,..., m^-_k)^T,\ \ \mathbf{m}^+=( m^+_1,..., m^+_k)^T,\\[1mm]
A=(a_{ij})_{k\times k},\ \ B=A-I_k,
\end{array}\rr.
\ees
where $I_k$ is the $k\times k$ identity matrix. Clearly, the diagonal entries  of $B$ are 0 because of \eqref{aij}.

To study the global stability of positive equilibrium solution of the problem \eqref{4.1}, we make the following assumptions:\vspace{-1mm}
\begin{enumerate}[leftmargin=4em]
\item[$\mathbf{(F_1)}$]  The determinant  $\det{[A(2I_k-A)]}\neq 0$,
and the algebraic equations
\bess
\begin{bmatrix}
A & 0\\
0 &  A
\end{bmatrix}\mathbf{c}_*^T=\begin{bmatrix}
\mathbf{m}^-+(2I_k-A)^{-1}(\mathbf{m}^+-\mathbf{m}^-)\\[.2mm]
\mathbf{m}^+-(2I_k-A)^{-1}(\mathbf{m}^+-\mathbf{m}^-)
\end{bmatrix}
\eess
has a unique positive solution $\mathbf{c}_*:=(\bar c_1,...,\bar c_k, \underline c_1,...,\underline c_k)\in\mathbb{R}^{2k}$.
\vspace{-2mm}
\item[$\mathbf{(F_{2})}$] There exist two $k\times k$ diagonal matrices $Q_1$, $Q_2$ with positive constant entries such that
 both $Q_1$ and $4Q_2-(Q_2 B+B^T Q_1){Q_1}^{-1}(B^T Q_2 +Q_1 B)$ are positive definite.
\item[$\mathbf{(F_{3})}$] There exists a $k\times k$ diagonal matrix $Q_3$ with positive constant entries, such that
$Q_3(I_k-B-\mathbf{c}_1)+(I_k-B-\mathbf{c}_1)^T Q_3$ is positive definite,
where
\bes\label{c1}
    \mathbf{c}_1=\text{diag}\bigg(\frac{\bar c_1-\underline c_1}{\bar c_1},\frac{\bar c_2-\underline c_2}{\bar c_2},...,\frac{\bar c_k-\underline c_k}{\bar c_k}\bigg),
\ees
  and $\bar c_i$ and $\underline c_i$ for $1\leq i\leq k$ are given by $\mathbf{(F_1)}$.
\end{enumerate}

We will prove that if the assumptions $\mathbf{(F_1)}$ and one of $\mathbf{(F_{2})}$ and $\mathbf{(F_{3})}$ are satisfied, then the elliptic problem \eqref{4a1} has a positive  solution and the system \eqref{4.1} is permanent. Especially, if $\mathbf{(F_{1})}$ and $\mathbf{(F_{3})}$ hold, then the following Theorem \ref{th4.6} shows that the positive equilibrium solution is unique and  globally asymptotically stable.

We first give the estimates of positive solutions $(u_1,...,u_k)$ of \eqref{4.1} by the upper and lower solutions method. Let $(\bar u_1(t),...,\bar u_k(t),\underline u_1(t),...,\underline u_k(t))$ be the
unique solution of
\begin{equation}\label{4a2}
\begin{cases}
\dd \bar u_i'=\bar u_i\bigg( m^+_i-\bar u_i-\sum_{1\leq j\leq k,\,j\neq i} a_{ij}\underline u_{j}\bigg), & t>0,\, i=1,..., k,\\[4mm]
\dd \underline u_i'=\underline u_i\bigg( m^-_i-\underline u_i-\sum_{1\leq j\leq k,\,j\neq i} a_{ij}\bar u_{j}\bigg), &t>0,\, i=1,..., k,\\[4mm]
\bar u_i(0)=\dd\max_{x\in\ol\Omega}\varphi_{i}(x),\ \underline u_i(0)=\dd\min_{x\in\ol\Omega}\varphi_{i}(x),&i=1,..., k.
 \end{cases}
\end{equation}
 Here, without loss of generality, we can assume $\phi_i(x)>0$ on $\ol\Omega$ since the solution $u_i(x,t)$ of  \eqref{4.1} is positive for any $t>0$ which can be easily obtained by applying upper and lower solutions method \cite[Theorem 8.1]{pcv1992} and  Hopf's  Lemma for parabolic equations.
Then $(\bar u_1(t),...,\bar u_k(t))$ and $(\underline u_1(t),...,\underline u_k(t))$ are a pair of coupled ordered upper and lower solutions of  \eqref{4.1} and
\begin{equation}\label{4.6}
0<\underline u_i(t)\leq u_i(x,t)\leq \bar u_i(t),\ \ \forall\;x\in\ol\Omega,\;t>0.
\end{equation}

Before giving the  estimates of the positive solutions of \eqref{4.1}, we  recall some preliminary results about matrices. For any $k\times k$ matrices $M$, $N$, $P$ and $R$, the following results hold (See, e.g., \cite[Page 104 and 149-150]{hees1990}):
\begin{align}
&\det{\small\Big[\begin{array}{ll}
M &N \\[-.2mm] P &R
\end{array}\Big]}=\det(M)\det(R-PM^{-1}N)= \det(R)\det(M-NR^{-1}P),\label{4.7} \\
&{\small\Big[\begin{array}{ll}
M &N \\[-.2mm] N^{T} &R
\end{array}\Big]}\ {\rm is\ positive\ definite} \ \Longleftrightarrow \ \mbox{both} \
M \ {\rm and} \ R-N^TM^{-1}N\ {\rm are\  positive\ definite},\label{4.8}\\
& {\rm If} \ M\ {\rm is\ positive\ definite, \ then} \ xM {x}^T\geq \ \var {x}{x}^T\ for\ all\ x\in \mathbb{R}^k\ and\ some\ \var>0.\label{4.9}
\end{align}

Especially when $N=N^{T}$ and $M=R$, we have the following elementary Lemma.
\begin{lemma}\label{lemma-positive}Let $M$ and $N$ be two $k\times k$ symmetric matrices. Then $M+N$ and $M-N$ are positive definite if and only if $H={\small\Big[\begin{array}{ll}
M & N\\[-.2mm] N &M
\end{array}\Big]}$ is positive definite.
 \vspace{-1mm}\end{lemma}

\begin{proof}
(i) Suppose that $M+N$ and $M-N$ are positive definite. We will show that $H$ is positive definite. Set $\mathbf{x}_1,\mathbf{x}_2\in \mathbb{R}^{k}$, $X=(\mathbf{x}_1,\mathbf{x}_2)$ and $X\neq 0$, then we have
\bes
X HX^T&=& \mathbf{x}_1 M \mathbf{x}_1^T+\mathbf{x}_2 M \mathbf{x}_2^T+2\mathbf{x}_1 N \mathbf{x}_2^T\nm\\[1mm]
&=& \mathbf{x}_1 (M+N) \mathbf{x}_1^T+\mathbf{x}_2 (M+N) \mathbf{x}_2^T+2\mathbf{x}_1 N \mathbf{x}_2^T-\mathbf{x}_1 N \mathbf{x}_1^T-\mathbf{x}_2 N \mathbf{x}_2^T\nm\\[1mm]
&=& \mathbf{x}_1 (M+N) \mathbf{x}_1^T+\mathbf{x}_2 (M+N) \mathbf{x}_2^T+\mathbf{x}_1 N (\mathbf{x}_2-\mathbf{x}_1)^T+(\mathbf{x}_1-\mathbf{x}_2) N \mathbf{x}_2^T\nm\\[1mm]
&=& \mathbf{x}_1 (M+N) \mathbf{x}_1^T+\mathbf{x}_2 (M+N) \mathbf{x}_2^T-(\mathbf{x}_2-\mathbf{x}_1)N (\mathbf{x}_2-\mathbf{x}_1)^T\nm\\[1mm]
&=& \mathbf{x}_1 \frac{M+N}{2} \mathbf{x}_1^T+\mathbf{x}_2 \frac{M+N}{2} \mathbf{x}_2^T-2\mathbf{x}_1 \frac{M+N}{2} \mathbf{x}_2^T+\mathbf{x}_1 \frac{M+N}{2} \mathbf{x}_1^T\nm\\[1mm]
&&+\mathbf{x}_2 \frac{M+N}{2} \mathbf{x}_2^T+2\mathbf{x}_1 \frac{M+N}{2} \mathbf{x}_2^T-(\mathbf{x}_2-\mathbf{x}_1)N (\mathbf{x}_2-\mathbf{x}_1)^T\nm\\[1mm]
&=&(\mathbf{x}_2-\mathbf{x}_1) \frac{M+N}{2} (\mathbf{x}_2-\mathbf{x}_1)^T+(\mathbf{x}_2+\mathbf{x}_1) \frac{M+N}{2}(\mathbf{x}_2+\mathbf{x}_1)^T\nm\\[1mm]
&&-(\mathbf{x}_2-\mathbf{x}_1)N (\mathbf{x}_2-\mathbf{x}_1)^T\nm\\[1mm]
&=&(\mathbf{x}_2-\mathbf{x}_1) \frac{M-N}{2} (\mathbf{x}_2-\mathbf{x}_1)^T+(\mathbf{x}_2+\mathbf{x}_1) \frac{M+N}{2} (\mathbf{x}_2+\mathbf{x}_2)^T>0.\label{4.9aa}
\ees

(ii) Assume that the matrix $H$ is positive definite. Suppose on the contrary that $M+N$ or $M-N$  is not positive definite. Without loose of generality, we assume that $M-N$  is not positive definite.  Then there exists $0\neq\mathbf{\tilde x}\in \mathbb{R}^k$ such that $ \mathbf{\td x}(M-N)\mathbf{\td x}^T\leq 0$. Let $\mathbf{\td X}=(- \mathbf{\td x}, \mathbf{\td x})$. From \eqref{4.9aa} we have
\bess
\mathbf{\td X} H \mathbf{\td X}^T&=&(\mathbf{\td x}+\mathbf{\td x}) \frac{M-N}{2} (\mathbf{\td x}+\mathbf{\td x})^T+(\mathbf{\td x}-\mathbf{\td x}) \frac{M+N}{2} (\mathbf{\td x}-\mathbf{\td x})^T\\[1mm]
&=& 4\mathbf{\td  x}\frac{M-N}{2} \mathbf{\td x}^T\leq 0,
\eess
which contradicts to the fact that $H$ is positive definite. Thus $M+N$ and $M-N$  are positive definite. The proof is finished.
\end{proof}

\begin{corollary}\label{coro4.3} If there exists a $k\times k$ diagonal matrix $Q_4$ with positive constant entries such that
  \bes\label{4.9a}
  Q_4(I_k-B)+(I_k-B)^T Q_4\ \ {\rm is\ positive\ definite},
  \ees
where $B$ is given by \eqref{4.3}. Then
$Q_5 {\small\Big[\begin{array}{ll}
I_k & B\\[-.2mm]
B & I_k
\end{array}\Big]}+{\small\Big[\begin{array}{ll}
I_k & B\\[-.2mm]
B & I_k
\end{array}\Big]}^TQ_5$ is positive definite, where $Q_5={\small\Big[\begin{array}{cc}
Q_4 & 0\\[-.2mm]
0 & Q_4
\end{array}\Big]}$ is a $2k\times 2k$ diagonal matrix having positive constant entries.
\end{corollary}

\begin{proof}
For $\mathbf{x}=(x_1,...,x_n),\ \hat{\mathbf{x}}=(|x_1|,...,|x_n|)$ and $\mathbf{x}\neq \mathbf{0}$. It then follows from \eqref{4.9a} that
\begin{equation*}
\begin{split}
\mathbf{x}Q_4(I+B)\mathbf{x}^T+\mathbf{x}(I+B)^T Q_4\mathbf{x}^T&=2\mathbf{x}Q_4\mathbf{x}^T+\mathbf{x}Q_4B\mathbf{x}^T+\mathbf{x}B^T Q_4\mathbf{x}^T\\
&\geq 2\mathbf{x}Q_4\mathbf{x}^T-\hat{\mathbf{x}}Q_4B\hat{\mathbf{x}}^T-\hat{\mathbf{x}}B^T Q_4\hat{\mathbf{x}}^T\\
&=\hat{\mathbf{x}}Q_4(I-B)\hat{\mathbf{x}}^T+\hat{\mathbf{x}}(I-B)^T Q_4\hat{\mathbf{x}}^T>0.
\end{split}
\end{equation*}
This, combined with Lemma \ref{lemma-positive}, implies the desired conclusion.
\end{proof}

We recall that the system \eqref{4.1} is uniformly persistent    (See, e.g., \cite[Page 3]{hw1989}) if all solutions satisfy $\dd\liminf_{\ttt}u_i(x,t)>0$  for all $1\leq i\leq k$ and $x\in \bar{\Om}$, and it is permanent (See, e.g., \cite{hs1992}) if it also satisfies $\dd\limsup_{t\to\yy}u_i(x,t)\leq M$ for some $M>0$. Now we prove the following result which concerns with  the permanence property of \eqref{4.1},
and also give estimates of positive solution $(u_1,...,u_k)$ of \eqref{4.1}.

\begin{proposition}\label{lemma4.4} Assume that $\mathbf{(F_1)}$ and one of $\mathbf{(F_{2})}$ or \eqref{4.9a} holds. Then the problem \eqref{4a2} has a unique positive equilibrium $\mathbf{c}_*=(\bar c_1,...,\bar c_k, \underline c_1,...,\underline c_k)\in\mathbb{R}^{2k}$ with $\bar c_i\geq \underline c_i$, and $\mathbf{c}_*$ is globally asymptotically stable with respect to \eqref{4a2}. Moreover, the solution $(u_1,...,u_k)$ of \eqref{4.1} satisfies
\bes\label{4.10}
0<\underline c_i\leq \liminf_{t\to\yy}u_i(x,t)\leq \limsup_{t\to\yy}u_i(x,t)\leq\bar c_i,\ \ \forall\;x\in\ol\Omega,\;t>0,\; 1\leq i\leq k,
\ees
which implies that the problem \eqref{4.1} is permanent. In particular the problem \eqref{4a1} has a positive solution $(u_1^*(x),...,u_k^*(x))$ which satisfies
\begin{equation}\label{4.14}
\begin{split}
0<\underline c_i \leq u_i^*(x)\leq\bar c_i,\ \ \forall\ x\in\ol\Omega,\; 1\leq i\leq k.
\end{split}
\end{equation}
\end{proposition}

\begin{proof} Since $B=A-I_k$, we have $I_k-B^2=I_k-(A-I_k)^2=A(2I_k-A)$. From $\mathbf{(F_1)}$, we have  $\det(I_k-B^2)=\det[A(2I_k-A)]\neq0$, then \eqref{4a2} has a unique equilibrium $\mathbf{c}_*\in\mathbb{R}^{2k}$ satisfying
\bes\label{4.10a}
 \begin{bmatrix}
I_k & B\\
B & I_k
\end{bmatrix}\mathbf{c}_*^T=\begin{bmatrix}\mathbf{m}^+\\
\mathbf{m}^-\end{bmatrix},
 \ees
where $I_k$, $B$, $\mathbf{m}^-$ and $\mathbf{m}^+$ are given by \eqref{4.3}. Denote
\bess
 B_1=\begin{bmatrix}
(I_k-B)^{-1} & -(I_k-B)^{-1}+I_k\\[1mm]
 -(I_k-B)^{-1}+I_k &  (I_k-B)^{-1}
\end{bmatrix}.
 \eess
It follows from \eqref{4.7} that $\det B_1=\det (I_k-B)^{-1}\det (I_k+B)\neq 0$.  Then multiplying the equation \eqref{4.10a} by $B_1$ on the left, we have
\bes\label{4.12}
\begin{bmatrix}
A & 0\\
0 &  A
\end{bmatrix}\mathbf{c}_*^T=\begin{bmatrix}
\mathbf{m}^-+(I_k-B)^{-1}(\mathbf{m}^+-\mathbf{m}^-)\\[1mm]
\mathbf{m}^+-(I_k-B)^{-1}(\mathbf{m}^+-\mathbf{m}^-)
\end{bmatrix}.
 \ees
Thanks to $A=I_k+B$ and $\mathbf{(F_1)}$, it yields $\underline c_i,\ \bar c_i>0$, i.e.,  each element of the vector $\mathbf{c}_*$ is positive.

When $\mathbf{(F_1)}$ and $\mathbf{(F_{2})}$ hold, let $Q_1$, $Q_2$ be given by $\mathbf{(F_{2})}$, and $Q=(Q_1, Q_2)$. Then $Q$ is a $2k\times 2k$ diagonal matrix with positive constant entries. The direct calculation yields
 \bess
E:=Q \begin{bmatrix}
I_k & B\\
B & I_k
\end{bmatrix}+\begin{bmatrix}
I_k & B\\
B & I_k
\end{bmatrix}^TQ=\begin{bmatrix}
2Q_1 & Q_1B+B^TQ_2\\[1mm]
Q_2B+B^TQ_1 & 2Q_2
\end{bmatrix}.
\eess
Owing to $\mathbf{(F_{2})}$ and \eqref{4.8} we see that the matrix $E$ is positive definite. When $\mathbf{(F_1)}$ and \eqref{4.9a} hold, let  $Q={\small\Big[\begin{array}{cc}
	Q_4 & 0\\[-.2mm]
	0 & Q_4
	\end{array}\Big]}$, where $Q_4$ is given by \eqref{4.9a}. Then $Q$ is a $2k\times 2k$ diagonal matrix with positive constant entries. By Corollary \ref{coro4.3} we know that the matrix $E$ is positive definite. In conclusion, the positive equilibrium $\mathbf{c}_*$ is globally asymptotically stable with respect to \eqref{4a2} by Theorem \ref{th4.1}. This, combined with \eqref{4.6}, implies \eqref{4.10}.

Since $(\bar c_1,...,\bar c_k, \underline c_1,...,\underline c_k)$, with $0<\underline c_i\leq \bar c_i$ for $1\leq i\leq k$, is the unique positive equilibrium of \eqref{4a2}, and so $(\bar c_1,...,\bar c_k)$ and $(\underline c_1,...,\underline c_k)$ are a pair of coupled ordered upper and lower solutions of \eqref{4a1}.
By \cite[Theorem 10.2, Page\,440]{pcv1992} we know that the problem \eqref{4a1} has a positive
solution and \eqref{4.14} holds.\end{proof}

Now we have the following result on the global stability of the positive equilibrium of \eqref{4.1}.
\begin{theorem}\label{th4.6}\, Let $\mathbf{(F_1)}$ and $\mathbf{(F_3)}$ hold.
Then the problem \eqref{4a1} has a unique positive solution $(u_1^*(x),\cdots,u_k^*(x))$, and  every solution of \eqref{4.1} satisfies $\dd\lim_{t\to\yy}u_i(x,t)=u_i^*(x)$ in $C^2(\ol\Omega)$ for $1\leq i\leq k$.
\end{theorem}

\begin{proof} Note that $\mathbf{(F_3)}$ implies \eqref{4.9a}. By Proposition \ref{lemma4.4},  the problem \eqref{4a1} has at least one positive solution, denoted by $(u_1^*(x),...,u_k^*(x))$. Let $(u_1,...,u_k)$ be the solution of  \eqref{4.1}.
 Define a functional $F: [0,\infty)\to \mathbb{R}$ by
$$F(t)=\dd\sum_{ i=1}^{k}{\varepsilon}_i\int_{\Om}\int_{u_i^*(x)}^{u_i(x,t)} u_i^*(x)\frac{s-u_i^*(x)}{s} {\rd}s{\rd}x,$$
where $\varepsilon_i>0$ will be determined later.  By \eqref{2.6}, \eqref{2.8b1} and \eqref{4.14},
\begin{equation*}
\begin{split}
\frac{dF}{dt}\leq&-\dd \int_{\Om}\bigg(\sum_{ i=1}^{k}\varepsilon_id_iu_i^2\bigg|\nabla \frac{u_i^*}{u_i}\bigg|^2+\sum_{1\leq i,j\leq k}\varepsilon_iu_i^*a_{ij} (u_i-u_i^*)(u_j-u_j^*)\bigg){\rd}x\\[.1mm]
\leq&\int_{\Om}\bigg(-\sum_{ i=1}^{k}\var_i \underline c_i(u_i-u_i^*)^2
+\sum_{i\neq j}\var_i\bar c_i a_{ij} |u_i-u_i^*|\cdot|u_j-u_j^*|\bigg){\rd}x\\[.1mm]
=& \int_{\Om} \bigg(-\sum_{ i=1}^{k} \var_i \bar c_i(u_i-u_i^*)^2
+\sum_{i\neq j}\varepsilon_i\bar c_i a_{ij} |u_i-u_i^*|\cdot|u_j-u_j^*|\bigg){\rd}x\\[.1mm]
&+\int_{\Om}\sum_{ i=1}^{k} \var_i (\bar c_i-\underline c_i)(u_i-u_i^*)^2{\rd}x\\[.1mm]
=&\dd -\frac{1}{2}\int_{\Om}\mathbf{U}_1[Q(I_k-B-\mathbf{c}_1)
+(I_k-B-\mathbf{c}_1)^TQ]\mathbf{U}_1^T{\rd}x,
\end{split}
\end{equation*}
where $\mathbf{c}_1$ is defined as in \eqref{c1},
 \bess
 \mathbf{U}_1=(|u_1-u_1^*|,...,|u_k-u_k^*|),\;\;
 Q=\text{diag}(\var_1 \bar c_1,\var_2 \bar c_2,...,\var_k \bar c_k).
  \eess
Take $\varepsilon_i>0$ such that $Q=Q_3$ is given by $\mathbf{(F_3)}$. It then follows from the
assumption $\mathbf{(F_3)}$ and fact \eqref{4.9} that there exists $\delta>0$ such that
\bess
\frac{dF}{dt}\leq -\frac{\delta}{2} \int_{\Om}\sum_{ i=1}^{k} (u_i-u_i^*)^2 {\rd}x\leq 0.
\eess
 Now from Theorem \ref{th2.1} and Lemma \ref{th2.2}, similar to the proof of Corollary  \ref{coro2.4}, we get $\dd\lim_{t\to\yy}u_i(x,t)=u_i^*(x)$  in $C^2(\ol\Omega)$
for $1\leq i\leq k$. The proof is finished.
\end{proof}
Note that the condition $\mathbf{(F_2)}$ is weaker than $\mathbf{(F_3)}$. With the condition $\mathbf{(F_2)}$, the system \eqref{4.1} is permanent and has a positive equilibrium, but it is not clear whether the positive equilibrium is unique and globally asymptotically stable. The condition $\mathbf{(F_3)}$ ensures the uniqueness and global stability of the positive equilibrium.
We give an application of Theorem \ref{th4.6} to more specific resource functions.
\begin{corollary}\label{coro4.8}
Assume  $m_i(x)=1+\varepsilon f_i(x)$ with $f_i\in C^{\alpha}(\ol\Omega)$ satisfies  $|f_i(x)|\leq 1$ on $\ol\Omega$.
If
\begin{enumerate}
\item [{\rm(i)}] $A$ is a symmetric and diagonally dominant matrix;

\item [{\rm(ii)}] The vector $A^{-1}\mathbf{v}^T$ has positive entries, where $\mathbf{v}=(1,1,...,1)\in \mathbb{R}^n$.
\end{enumerate}
Then there exists a positive constant $\varepsilon_0$ such that for any $0<\varepsilon<\varepsilon_0$,  the problem \eqref{4a1} has a unique positive solution which is globally asymptotically stable with respect to the problem \eqref{4.1}.
\end{corollary}

\begin{proof} We first show that $0<\underline c_i\leq \bar c_i$ for $1\leq i\leq k$,  where $\underline c_i$ and $\bar c_i$ are defined as in $\mathbf{(F_1)}$. Recalling that the  diagonal entries of $B$ are 0. By the condition (i), both $A=I_k+B$ and $I_k-B$ are positive definite and so are non-degenerate. Denote $\bar{\mathbf{c}}=(\bar c_1,...,\bar c_k)$, $\underline{\mathbf{c}}=(\underline c_1,...,\underline c_k)$ and $\mathbf{c}_*=(\bar{\mathbf{c}},\underline{\mathbf{c}})$. Then   from \eqref{4.12}, we obtain
\bess
\begin{bmatrix}
A & 0\\
0 &  A
\end{bmatrix}\mathbf{c}_*^T&=&
\begin{bmatrix}
\mathbf{m}^-+(I_k-B)^{-1}(\mathbf{m}^+-\mathbf{m}^-)\\[.5mm]
\mathbf{m}^+-(I_k-B)^{-1}(\mathbf{m}^+-\mathbf{m}^-)
\end{bmatrix}\\
&=&\begin{bmatrix}
\frac 1 2(\mathbf{m}^++\mathbf{m}^-)-\frac 1 2(\mathbf{m}^+-\mathbf{m}^-)+(I_k-B)^{-1}(\mathbf{m}^+-\mathbf{m}^-)\\[.5mm]
\frac 1 2(\mathbf{m}^++\mathbf{m}^-)+\frac 1 2(\mathbf{m}^+-\mathbf{m}^-)-(I_k-B)^{-1}(\mathbf{m}^+-\mathbf{m}^-)
\end{bmatrix}\\[1mm]
 &=&\frac{1}{2}\begin{bmatrix}
\mathbf{m}^++\mathbf{m}^-+A(I_k-B)^{-1}(\mathbf{m}^+-\mathbf{m}^-)\\[.5mm]
\mathbf{m}^++\mathbf{m}^--A(I_k-B)^{-1}(\mathbf{m}^+-\mathbf{m}^-)
\end{bmatrix},
\eess
where $A=I_k+B$, and so
\bes\label{4.17aa}\lf\{\begin{array}{ll}
	\bar{\mathbf{c}}=\frac{1}{2}A^{-1}(\mathbf{m}^-+\mathbf{m}^+)
	+\frac{1}{2}(I_k-B)^{-1}(\mathbf{m}^+-\mathbf{m}^-),\\[1mm]
		\underline{\mathbf{c}}=\frac{1}{2}A^{-1}(\mathbf{m}^-+\mathbf{m}^+)
	-\frac{1}{2}(I_k-B)^{-1}(\mathbf{m}^+-\mathbf{m}^-),\\[1mm]
\bar{\mathbf{c}}-\underline{\mathbf{c}}=(I_k-B)^{-1}(\mathbf{m}^+-\mathbf{m}^-).
\end{array}\rr.
\ees
Since $m_i(x)=1+\varepsilon f_i(x)$ with $-1 \leq f_i(x)\leq 1$ on $\ol\Omega$, we have $\mathbf{m}^+=(1+\varepsilon)\mathbf{v}^T$, $\mathbf{m}^-=(1-\varepsilon)\mathbf{v}^T$,  where $\mathbf{v}=(1,1,...,1)$. Then by \eqref{4.17aa},
\bes\label{4.18aa}\lf\{\begin{array}{ll}
 \bar{\mathbf{c}}=A^{-1}\mathbf{v}+\varepsilon(I_k-B)^{-1}\mathbf{v}^T,\\[1mm]  \underline{\mathbf{c}}=A^{-1}\mathbf{v}-\varepsilon(I_k-B)^{-1}\mathbf{v}^T,\\[1mm]
 \bar{\mathbf{c}}-\underline{\mathbf{c}}=2\varepsilon(I_k-B)^{-1}\mathbf{v}^T.
 \end{array}\rr.
 \ees
Thanks to the condition (ii), the first two equalities of \eqref{4.18aa} and $0<\varepsilon\ll 1$, it follows that $\bar c_i,\underline c_i>0$ for $1\leq i\leq k$. Moveover we claim that
\bes\label{4.19aa}
\underline c_i\leq \bar c_i,\ \ \ \forall\,1\leq i\leq k.
\ees
Notice that $I_k-B$ is positive definite.  Taking advantages of  Lemma \ref{lemma4.4} and \eqref{4.9a} replacing $Q_4$ by $I_k$ in there, the inequality \eqref{4.19aa} is derived.

Using \eqref{4.18aa} and \eqref{4.19aa}, we see that ${\bar c_i-\underline c_i}>0$ and ${\bar c_i}$ for $1\leq i\leq k$ are linear increasing with respect to $\varepsilon$. Meanwhile,
it can be verified that $\dd\frac{\bar c_i-\underline c_i}{\bar c_i}$ for $1\leq i\leq k$ are linear increasing with respect to $0<\varepsilon\ll 1$. Recalling that the matrix $I_k-B$ is positive definite,  by \eqref{4.8} we get the positive definiteness of the matrix $I_k-B-\mathbf{c}_1$ for $0<\varepsilon<\varepsilon_0$ provided $\varepsilon_0>0$ is small,  where $ \mathbf{c}_1$ is defined in \eqref{c1}. The desired conclusion is followed by Theorem \ref{th4.6}. The proof is finished.
\end{proof}

The global stability of the positive coexistence state in Corollary \ref{coro4.8} is achieved under a weak competition condition on the competition matrix $A$ ($A$ is diagonally dominant) and the resource function being a small perturbation from homogeneous one. We end this subsection by giving another two examples of competition with $2$ and $4$ species.

\begin{example} Let $k=2$ and
\begin{equation*}
A=\begin{bmatrix}
1 & a_{12}\\
a_{21} & 1
\end{bmatrix},\ \ \
B=\begin{bmatrix}
0 & a_{12}\\
a_{21} & 0
\end{bmatrix},\ \ \
\mathbf{m}^-=\begin{pmatrix}
m_1^- \\
m_2^-
\end{pmatrix},\ \ \
\mathbf{m}^+=\begin{pmatrix}
m_1^+ \\
m_2^+
\end{pmatrix}.
\end{equation*}
Then the conclusions in Theorem \ref{th4.6} hold  if
\begin{align}
&a_{12}<\frac{m_1^-}{m_2^+}\leq \frac{m_1^+}{m_2^-}<\frac{1}{a_{21}},\label{4.20b}\\
&a_{12}a_{21}<\lf( 1-\frac{m_1^+-m_1^-+a_{12}(m_2^+-m_2^-)}{m_1^+-a_{12}m_2^-}\rr) \lf( 1-\frac{a_{21}(m_1^+-m_1^-)+m_2^+-m_2^-}{m_2^+-a_{21}m_1^-}\rr).\label{4.21b}
\end{align}
We verify $\mathbf{(F_{1})}$ and $\mathbf{(F_{3})}$ under the conditions \eqref{4.20b} and \eqref{4.21b}. A simple calculation gives
\begin{equation*}
A^{-1}=\frac{1}{1-a_{12}a_{12}}\begin{bmatrix}
1 & -a_{12}\\
-a_{21} & 1
\end{bmatrix},\ \ \ \ \ \
(I_2-B)^{-1}=\frac{1}{1-a_{12}a_{21}}\begin{bmatrix}
1 & a_{12}\\
a_{21} & 1
\end{bmatrix}.
\end{equation*}
Then from \eqref{4.17aa}, we see
\begin{align*}
&\bar{\mathbf{c}}=\frac{1}{1-a_{21}a_{12}}(m_1^+-a_{12}m_2^-,m_2^+-a_{21}m_1^-)^T,\\
&\underline{\mathbf{c}}=\frac{1}{1-a_{21}a_{12}}(m_1^--a_{12}m_2^+,m_2^--a_{21}m_1^+)^T,\\
&\bar{\mathbf{c}}-\underline{\mathbf{c}}=\frac{1}{1-a_{21}a_{12}}(m_1^+-m_1^-+a_{12}(m_2^+-m_2^-),a_{21}(m_1^+-m_1^-)+m_2^+-m_2^-)^T.
\end{align*}
By \eqref{4.20b}, any element in the vectors $\bar{\mathbf{c}}$, $\underline{\mathbf{c}}$  is positive,  and each element in $\bar{\mathbf{c}}-\underline{\mathbf{c}}$  is nonnegative, which implies that  $\mathbf{(F_{1})}$ holds.  Using the above formulas, we deduce
\begin{equation*}
I_k-B-\mathbf{c}_1=
\begin{bmatrix}
1-\dd\frac{m_1^+-m_1^-+a_{12}(m_2^+-m_2^-)}{m_1^+-a_{12}m_2^-}&-a_{12}\\
-a_{21}&1-\dd\frac{a_{21}(m_1^+-m_1^-)+m_2^+-m_2^-}{m_2^+-a_{21}m_1^-}
\end{bmatrix},
\end{equation*}
where $\mathbf{c}_1$ is defined in \eqref{c1}.  It is well known that  $\mathbf{(F_{3})}$ holds if and only if \eqref{4.21b} is satisfied.

If both $m_1$ and $m_2$ are positive constants, then the two conditions \eqref{4.20b} and \eqref{4.21b} become $a_{12}<\dd\frac{m_1}{m_2}<\dd\frac{1}{a_{21}}$ which  coincides with the weak competition condition in the two species diffusive competitive problem in an homogeneous environment \cite{bp1983,gbs1977}. On the other hand, for the nonhomogeneous environment case, the result here is not as optimal as the ones in \cite{hn2016}, but our proof is completely different: we use Lyapunov functional method, and we do not use the monotone dynamical system method.
\end{example}

\begin{example}\label{ex2} Suppose that $m_i(x)$ for $1\leq i\leq k$ satisfy the condition in Corollary \ref{coro4.8}. If $k=4$ and
\bess
A=\begin{bmatrix}
 1  \ \  \ \ & 0.2 \ \ \ \ & 0.1 \ \ \ \ & 0.1\\
 0.2   & 1     & 0.2   &  0.15\\
 0.1   & 0.2   & 1     &  0.1\\
 0.1   & 0.15  & 0.1   &  1
\end{bmatrix},
\eess
then for $0<\varepsilon\leq 0.1$, the results in Corollary \ref{coro4.8} hold true.
\end{example}

\subsection{Global stability of semi-trivial equilibrium solutions}\label{s4.2}

Without loss of generality, we investigate the global stability of the semi-trivial equilibrium solution with the form $\mathbf{u}_{i_0}^{*}:=(u_1^*,...,u_{i_0}^{*},0,...,0)$,
where $1\leq i_0\leq k-1$, and ${\mathbf{u}}_{i_0}^{**}:=(u_1^*,...,u_{i_0}^*)$
is the positive solution of the following elliptic problem
\begin{equation}\label{4.14a}
\begin{cases}
\dd d_i\Delta u_i+u_i\bigg(m_i(x)-\sum_{ j=1}^{i_0} a_{ij}u_j\bigg)=0, & x\in \Omega,\; 1\leq i\leq i_0,\\[1.5mm]
\dd \partial_\nu u_i=0,& x\in \partial\Omega,\; 1\leq i\leq i_0.
 \end{cases}
\end{equation}

For the simplicity of notations, similar to \eqref{4.3}, we denote
\bes\label{4.14b}
A_{i_0}=(a_{ij})_{i_0\times i_0},\  B_{i_0}=A_{i_0}-I_{i_0},\  \widetilde{\mathbf{m}}^-=(m^-_1,..., m^-_{i_0})^T,\
\widetilde{\mathbf{m}}^+=( m^+_1,..., m^+_{i_0})^T,
\ees
where $I_{i_0}$ is the $i_0\times i_0$ identity matrix. The diagonal entries of $B_{i_0}$ are $0$.

To study the global stability of the semi-trivial equilibrium solution $\mathbf{u}_{i_0}^{*}$ of the problem \eqref{4.1}, we make the following assumptions:
\begin{enumerate}[leftmargin=4em]
\item[$\mathbf{(G_1)}$]
The determinant $\det[A_{i_0}(2I_{i_0}-A_{i_0})]\neq 0$, and the algebraic equations
\bes
\begin{bmatrix}
A_{i_0} & 0\\
0 &  A_{i_0}
\end{bmatrix}(\mathbf{c}_{i_0}^*)^T=\begin{bmatrix}
\wtd{\mathbf{m}}^-+(I_{i_0}-B_{i_0})^{-1}(\wtd{\mathbf{m}}^+-\wtd{\mathbf{m}}^-)\\[.5mm]
\wtd{\mathbf{m}}^+-(I_{i_0}-B_{i_0})^{-1}(\wtd{\mathbf{m}}^+-\wtd{\mathbf{m}}^-)
\end{bmatrix}\label{xx}
 \ees
has a unique positive solution $\mathbf{c}_{i_0}^*:=(\bar c_1,...,\bar c_{i_0}, \underline c_1,...,\underline c_{i_0})$ and
  \bes
  m^+_i-\sum_{ j=1}^{i_0} a_{ij}\underline c_i<0,\ \ \dd m^-_i-\sum_{ j=1}^{i_0} a_{ij}\bar c_i<0, \ \ \forall\;i_0+1\leq i\leq k,
  \label{yy}\ees
where $I_{i_0}$, $B_{i_0}$, $\widetilde{\mathbf{m}}^-$ and
$\widetilde{\mathbf{m}}^+$ are given by \eqref{4.14b}.
 \item[$\mathbf{(G_2)}$] There exists an $i_0\times i_0$ diagonal matrix $Q_6$ with positive constant entries such that
$Q_6(I_{i_0}-B_{i_0}-\mathbf{c}_2)+(I_{i_0}-B_{i_0}-\mathbf{c}_2)^T Q_6$ is positive definite, where
    \bes\label{c2}
    \mathbf{c}_2=\text{diag}\bigg(\frac{\bar c_1-\underline c_1}{\bar c_1},\frac{\bar c_2-\underline c_2}{\bar c_2},...,\frac{\bar c_{i_0}-\underline c_{i_0}}{\bar c_{i_0}}\bigg)
    \ees
and $\bar c_i$, $\underline c_i$ for $1\leq i\leq i_0$ are given by $\mathbf{(G_1)}$.
\end{enumerate}

\begin{lemma}\label{lemma4.7}
If $\mathbf{(G_1)}$ and one of  $\mathbf{(F_{2})}$ or \eqref{4.9a} hold.
\begin{enumerate}

\item[{\rm(i)}] The problem \eqref{4a2} has a semi-trivial equilibrium $\mathbf{c}_{i_0}^{**}=(\bar c_1,...,\bar c_{k}, \underline c_1,...,\underline c_{k})$ with $\bar c_i,\underline c_i>0$ for $1\leq i\leq i_0$ and  $\bar c_i=\underline c_i=0$ for $i_0+1\leq i\leq k$. The equilibrium $\mathbf{c}_{i_0}^{**}$ is globally asymptotically stable with respect to \eqref{4a2}. Moreover, the solution $(u_1,...,u_k)$ of \eqref{4.1} satisfies
\bes\label{4.15}
\lf\{\begin{array}{ll}
\underline c_i\leq\dd\liminf_{t\to\yy}u_i(x,t)\leq \limsup_{t\to\yy}u_i(x,t)\leq\bar c_i,\ \ &\forall\;x\in\ol\Omega,\;t>0,\; 1\leq i\leq i_0,\\[2mm]
\dd\lim_{t\to\yy}u_i(x,t)=0\ {\rm uniformly\ on}\ \ol\Omega, \ \ &\forall\; i_0+1 \leq i\leq k.
\end{array}\rr.
\ees

\item[{\rm(ii)}] For any $i_0+1\leq i\leq k$, we have
  \bes\label{4.16}
  \int_{0}^\yy\int_{\Omega}u^2_i(x,t){\rd}x{\rd}t< \yy.
  \ees

 \item[{\rm (iii)}]    The problem \eqref{4.14a} has a positive solution ${\mathbf{u}}_{i_0}^{**}=(u_1^*,u_2^*,...,u_{i_0}^*)$ which satisfies
 \begin{equation}\label{4.17}
 \begin{split}
 \underline c_i \leq u_i^*(x)\leq\bar c_i,\ \ \forall\ x\in\ol\Omega,\; 1\leq i\leq i_0.
 \end{split}
 \end{equation}
\end{enumerate}
\end{lemma}

\begin{proof} (i) As $\mathbf{(G_1)}$ holds, similar to the proof of Proposition \ref{lemma4.4} we know that the unique positive solution ${\mathbf{c}}_{i_0}^*$ of \eqref{xx} is the unique positive equilibrium of \eqref{4a2} with $k=i_0$.
Again by similar discussion as in the proof of Proposition \ref{lemma4.4}, there exists  $Q=\text{diag}({\var_1,...,\var_k,\eta_1,...,\eta_k})$ with $\var_i,\eta_i>0$ such that
 \bes\label{4.24aa}
Q{\small\Big[\begin{array}{ll}
I_k & B\\
B & I_k
\end{array}\Big]}+{\small\Big[\begin{array}{ll}
I_k & B\\
B & I_k
\end{array}\Big]}^TQ\ \ {\rm is\ positive\ definite},
\ees
where $B$ is given by \eqref{4.3}.

Let $(\bar u_1,...,\bar u_k, \underline u_1,...,\underline u_k)$ be
the solution of \eqref{4a2}. Define a function $F(t)$ by
\bess
F(t)=\dd\sum_{ i=1}^{i_0} \lf(\var_i\int_{\bar c_i}^{\bar u_i(t)}\frac{s-\bar c_i}{s}\rd s+\eta_i\int_{\underline c_i}^{\underline u_i(t)}\frac{s-\underline c_i}{s} \rd s\rr)+\sum_{ i=i_0+1}^{k}[\var_i \bar u_i(t)+\eta_i \underline u_i(t)].
\eess
Notice $\bar c_i=\underline c_i=0$ for $i_0+1\leq i\leq k$ and $a_{ii}=1$, $1\leq i\leq k$. From \eqref{4a2}, a direct computation yields
 \bess
\frac{dF}{dt}&=&\sum_{ i=1}^{i_0} \lf(\var_i \frac{\bar u_i-\bar c_i}{\bar u_i}\bar u_i'(t)+\eta_i\frac{\underline u_i-\underline c_i}{\underline u_i} \underline u_i'(t)\rr)+\sum_{ i=i_0+1}^{k}[\var_i \bar u_i'(t)+\eta_i\underline u_i'(t)]\\
&=&-\dd\sum_{1\leq i\leq i_0}
[\var_i(\bar u_i-\bar c_i)^2+\eta_i(\underline u_i-\underline c_i)^2]\\[.1mm]
&&+\sum_{1\leq i\leq i_0}\sum_{j\neq i}\big[\var_i a_{ij}(\bar u_i-\bar c_i)(\underline u_j-\underline c_j)+\eta_ia_{ij}(\underline u_i-\underline c_i)(\bar u_j-\bar c_j)\big]\\
&&+\dd\sum_{i_0+1\leq i\leq k}\var_i \bar u_i\bigg(m^+_i-\sum_{1\leq j\leq i_0}a_{ij}\underline c_j-\dd\sum_{j\neq i_0}a_{ij}(\underline u_i-\underline c_i)-\bar u_i\bigg)\\[.1mm]
&&+\dd\sum_{i_0+1\leq i\leq k}\eta_i \underline u_i\bigg(m^-_i-\sum_{1\leq j\leq i_0}a_{ij}\bar c_j-\dd\sum_{j\neq i_0}a_{ij}(\bar u_i-\bar c_i)-\underline u_i\bigg)\\[.1mm]
&=&-\dd\sum_{j\neq i}\big[\var_ia_{ij}(\bar u_i-\bar c_i)(\underline u_j-\underline c_j)+\eta_ia_{ij}(\underline u_i-\underline c_i)(\bar u_j-\bar c_j)\big]\\
&&-\dd\sum_{1\leq i,j\leq k}\big[\var_i(\bar u_i-\bar c_i)^2+\eta_i(\underline u_i-\underline c_i)^2\big]+\sum_{i_0+1\leq i\leq k} \var_i\bar u_i\bigg( m^+_i-\sum_{1\leq j\leq i_0}a_{ij}\underline c_j\bigg)\\[.1mm]
&&+\dd\sum_{i_0+1\leq i\leq k}\eta_i \underline u_i\bigg( m^-_i-\sum_{1\leq j\leq i_0}a_{ij}\bar c_j\bigg).
 \eess
From \eqref{yy} and \eqref{4.24aa},  it follows that
\begin{equation*}
\begin{split}
\frac{dF}{dt}\leq &-\dd\sum_{j\neq i}\big[\var_i a_{ij}(\bar u_i-\bar c_i)(\underline u_j-\underline c_j)+\eta_ia_{ij}(\underline u_i-\underline c_i)(\bar u_j-\bar c_j)\big]\\
&-\dd\sum_{1\leq i,j\leq k}\big[\var_i(\bar u_i-\bar c_i)^2+\eta_i(\underline u_i-\underline c_i)^2\big]\\[1mm]
=&-\frac{1}{2}(\mathbf{U}-\mathbf{c}_{i_0}^{**})\bigg(Q\begin{bmatrix}
I_k & B\\
B & I_k
\end{bmatrix}+\begin{bmatrix}
I_k & B\\
B & I_k
\end{bmatrix}^TQ\bigg)(\mathbf{U}-\mathbf{c}_{i_0}^{**})^{T}\le 0,
\end{split}
\end{equation*}
and
\bess
 \frac{dF}{dt}<0 \ \ \ {\rm if} \, \ \mathbf{U}\not=\mathbf{c}_{i_0}^{**},
 \eess
  where $\mathbf{U}=(\ol u_1,...,\ol u_k,\underline u_1,...,\underline u_k)$ and $Q=\text{diag}({\var_1,...,\var_k,\eta_1,...,\eta_k})$.
By the Lyapunov-LaSalle invariance principle, $\mathbf{c}_{i_0}^{**}$ is globally asymptotically stable with respect to the problem \eqref{4a2}. This combined with  \eqref{4.6} allows us to derive \eqref{4.15}.

(ii) It follows from  \eqref{4.1},  \eqref{yy} and \eqref{4.15} that there exist two constants $T>0$ and $0<\varepsilon\ll 1$ such that  for $i_0+1\leq i\leq k$ and $t\geq T$,
\begin{equation*}
\begin{split}
\dd\int_{T}^\yy\int_{\Omega} \frac{\partial u_i(x,t)}{\partial t}{\rd}x{\rd}t=&\int_{T}^\yy\int_{\Omega} u_i\bigg(m_i(x)-\sum_{1\leq j\leq k} a_{ij}u_j\bigg){\rd}x{\rd}t\\[.1mm]
  \leq&\int_{T}^\yy\int_{\Omega} u_i\bigg( m^+_i+\varepsilon-\sum_{1\leq j\leq i_0} a_{ij}\underline c_j-\sum_{i_0+1\leq j\leq k} a_{ij}u_j\bigg) {\rd}x{\rd}t\\[.1mm]
  \leq& -\int_{T}^\yy\int_{\Omega}u_i^2 {\rd}x{\rd}t \ \ {\rm (as} \  a_{ii}=1{\rm)},
\end{split}
\end{equation*}
  which implies that
  \bess
  \int_{T}^\yy\int_{\Omega}u_i^2 {\rd}x{\rd}t\leq\int_{\Omega} u_i(x,T) {\rd}x<\yy, \ \ \forall\;i_0+1\leq i\leq k.
  \eess
  Therefore \eqref{4.16} holds.

  (iii) From (i), $(\bar c_1,...,\bar c_{i_0}, \underline c_1,...,\underline c_{i_0})$ with $0<\underline c_i\leq \bar c_i$ for $1\leq i\leq i_0$ is the unique positive equilibrium of the problem \eqref{4a2} with $k=i_0$. Hence, $(\bar c_1,...,\bar c_{i_0})$ and $(\underline c_1,...,\underline c_{i_0})$ are a pair of ordered upper and lower solutions of the problem \eqref{4.14a}. By \cite[Theorem 10.2, {Page 440}]{pcv1992}, the problem \eqref{4.14a} has a positive solution $(u_1^*,...,u_{i_0}^*)$ and \eqref{4.17} holds.
\end{proof}

Now we prove the global stability of the semitrivial equilibrium.

\begin{theorem}\label{th4.11}
Suppose that the assumptions $\mathbf{(G_1)}$, $\mathbf{(G_2)}$ and  one of $\mathbf{(F_{2})}$ or \eqref{4.9a} holds. Let $(u_1^*,...,u_{i_0}^*)$ be the positive solution of \eqref{4.14a} in Lemma {\rm\ref{lemma4.7}}. Then  any solution of \eqref{4.1} satisfies
\begin{equation}\label{4.26aa}
\begin{cases}
\dd\lim_{t\to\yy}u_i(x,t)=u_i^*(x) \ \ \ &{\rm in} \ \ C^2(\ol\Omega),\ \ 1\leq i\leq  i_0,\\
\dd\lim_{t\to\yy}u_i(x,t)=0 \ \ \ &{\rm in} \ \ C^2(\ol\Omega),\ \ i_0+1\leq i\leq k.
\end{cases}
\end{equation}
\end{theorem}

\begin{proof} Let $(u_1,...,u_k)$ be the solution of \eqref{4.1}. We have known $\dd\lim_{t\to\yy}u_i(x,t)=0$ uniformly on $\ol\Omega$ for $i_0+1\leq i\leq k$ by the second result of \eqref{4.15}. Thanks to Theorem \ref{th2.1}, $u_i(\cdot,t)$ is uniformly bounded in $C^{2+\alpha}(\ol\Omega)$ for all $t\geq 1$ and some $0<\alpha<1$, which leads to the convergence of $u_i$ $(i_0+1\leq i\leq k)$ in $C^2(\ol\Omega)$.  Hence, in the following, we just  prove the first equation of \eqref{4.26aa}.

Define a function $F: [0,\infty)\to \mathbb{R}$ by
$$F(t)=\dd\sum_{ i=1}^{i_0} \bigg({\varepsilon}_i \int_{\Om}\int_{u_i^*(x)}^{u_i(x,t)} u_i^*(x)\frac{s-u_i^*(x)}{s} {\rd}s{\rd}x\bigg),$$
where $\varepsilon_i>0$ will be determined latter. Then $F(t)\ge 0$. It follows from \eqref{2.6}, \eqref{2.8b1}, \eqref{4.1} and \eqref{4.17} that
\begin{equation*}
\begin{split}
\frac{dF}{dt}=&\dd\sum_{1\leq i\leq i_0}\bigg[\varepsilon_i\int_{\Om} u_i^* (u_i-u_i^*)\bigg(-\sum_{1\leq j\leq i_0}a_{ij}(u_j-u_j^*)-\sum_{i_0+1\leq j\leq k}a_{ij}u_j\bigg){\rd}x\bigg]\\[1mm]
&-\int_{\Om}\sum_{1\leq i\leq i_0}\varepsilon_id_iu_i^2\Big|\nabla\frac{u_i^*}{u_i}\Big|^2 {\rd}x\\[1mm]
\leq&\dd -\int_{\Om}\sum_{1\leq i,j\leq i_0}\varepsilon_iu_i^*a_{ij} (u_i-u_i^*)(u_j-u_j^*){\rd}x\\[1mm]
&+\dd \int_{\Om}\bigg(\sum_{1\leq i\leq i_0}\frac{(k-i_0)\delta}{2} (u_i-u_i^*)^2+ \sum_{1\leq i\leq i_0}\frac{(\var_i\bar c_i)^2}{2\delta}\sum_{i_0+1\leq j\leq k}(a_{ij}u_j)^2\bigg){\rd}x,
\end{split}
\end{equation*}
where $0<\delta\ll 1$. Remember that $a_{ii}=1$. By the similar discussion as in the proof of Theorem \ref{th4.6}, we can derive
\begin{equation*}
\begin{split}
\frac{dF}{dt}\leq& \dd -\frac{1}{2}\int_{\Om}{\mathbf{U}}_2 [Q(I_{i_0}-B_{i_0}-\mathbf{c}_2)+(I_{i_0}-B_{i_0}-\mathbf{c}_2)^TQ]{\mathbf{U}}_2^T{\rd}x\\[1mm]
&+\dd \int_{\Om}\bigg(\sum_{1\leq i\leq i_0}\frac{(k-i_0)\delta }{2} (u_i-u_i^*)^2+ \sum_{1\leq i\leq i_0}\frac{(\var_i\bar c_i)^2}{2\delta}\sum_{i_0+1\leq j\leq k}(a_{ij}u_j)^2\bigg){\rd}x,
\end{split}
\end{equation*}
where $\mathbf{c}_2$ is defined in \eqref{c2},
 \bess
 {\mathbf{U}}_2=(|u_1-u_1^*|,...,|u_{i_0}-u_{i_0}^*|),\;\;\;
 Q=\text{diag}(\var_1 \bar c_1,...,\var_{i_0} \bar c_{i_0}).
  \eess
Take $\varepsilon_i>0$ such that $Q=Q_6$ in $\mathbf{(G_2)}$. It then follows from the assumption $\mathbf{(G_2)}$ and \eqref{4.9} that there exists $\delta_1>2{(k-i_0)\delta}>0$ such that
\begin{align*}
\frac{dF}{dt}\leq&\dd \int_{\Om}\bigg(\sum_{1\leq i\leq i_0}\lf[-\frac{\delta_1}{2}+\frac{(k-i_0)\delta }{2}\rr] (u_i-u_i^*)^2+ \sum_{1\leq i\leq i_0}\frac{(\var_i\bar c_i)^2}{2\delta}\sum_{i_0+1\leq j\leq k}(a_{ij}u_j)^2\bigg){\rd}x\\
\leq &-\frac{\delta_1}{4}\sum_{1\leq i\leq i_0}\int_{\Om}(u_i-u_i^*)^2 {\rd}x+\dd\sum_{1\leq i\leq i_0}\frac{(\var_i\bar c_i)^2}{2\delta}\sum_{i_0+1\leq j\leq k}a^2_{ij}\int_{\Om}u_j^2{\rd}x=:\psi(t) + h(t).
\end{align*}
Making use of Theorem \ref{th2.1}, Lemma \ref{th2.2} and \eqref{4.16}, by the similar arguments  as  in the proof of Corollary  \ref{coro2.4} we can obtain \eqref{4.26aa}. The proof is finished.
\end{proof}

In the following, we concern with the case $i_0=2$ and investigate the gloabel stability of semi-trivial equilibrium solution  $\mathbf{u}_{2}^{*}:=(u_1^*,u_{2}^{*},0,...,0)$  under a weaker condition than $\mathbf{(G_2)}$  by using the results in \cite[Theorem 1.4]{hn2016}.

For the convenience of readers, we  briefly recall the global stability results in \cite[Theorem 1.4]{hn2016}.  As $m_i(x)>0$ on $\ol\Omega$, we can define
\bes\label{3.4}
\left\{\begin{array}{lll}
	L_1=L_{d_1,m_2,m_1}=\dd\inf_{d_1>0}\frac{\ol m_2}{\ol{\theta}_{d_1,m_1,1}},\ \ S_1=S_{d_1,m_2,m_1}=\dd\sup_{d_1>0}\sup_{\ol\Omega}
	\frac{m_2}{{\theta_{d_1,m_1,1}}},\\[4.5mm]
	L_2=L_{d_2,m_1,m_2}=\dd\inf_{d_2>0}\frac{\ol m_1}{\ol{\theta}_{d_2,m_2,1}},\ \ S_2=S_{d_2,m_1,m_2}=\dd\sup_{d_2>0}\sup_{\ol\Omega}\frac{m_1}{{\theta}_{d_2,m_2,1}},
\end{array}\rr.
\ees
where $\theta_{d_i,m_i,1}$ is the unique positive solution of \eqref{3a} with $d=d_i$,
$m=m_i$ and $\varphi=1$, and $\ol m_i$ and $\ol{\theta}_{d_i,m_i,1}$ are defined as following
\bess
\ol m_i=\int_{\Omega} m_i(x) {\rm d}x,\ \ol{\theta}_{d_i,m_i,1}=\int_{\Omega} \theta_{d_i,m_i,1}(x) {\rm d}x,\  \ \ i=1,2.
\eess
Obviously, $L_1, L_2, S_1, S_2\in[0,\yy]$.

\begin{theorem}\label{th3.1}{\rm (\hspace{-.1mm}\cite[Theorem 1.4]{hn2016})} Assume that $d_i, a_{ij}$, $i,j=1,2$, are all positive constants. If at least one of $m_1$ and $m_2$ is nonconstant,
	then we have
	\bes\label{3.4b}
	0\leq L_1L_2<1,\ L_1S_2>1,\ L_2S_1>1,
	\ees
	and the following conclusions  hold:

\begin{enumerate}
	
	\item[{\rm (i)}] If $a_{21}/a_{11}\geq S_1$ and ${a_{12}}/{a_{22}}\leq {1}/{S_1}$, then for all $d_1$, $d_2>0$, $(\theta_{d_1,m_1,a_{11}},0)$ is globally asymptotically stable.
	
	\item[{\rm (ii)}] If ${a_{12}}/{a_{22}}\geq {S_2}$ and $a_{21}/a_{11}\leq {1}/{S_2}$, then for all $d_1$,  $d_2>0$, $(0,\theta_{d_2,m_2,a_{22}})$ is globally asymptotically stable.
	
	\item[{\rm (iii)}] If $a_{21}/a_{11}<L_1$ and $a_{12}/a_{22}<L_2$, then for all $d_1$, $d_2>0$, the problem \eqref{4.1} has a unique positive equilibrium solution that is globally asymptotically stable.
\end{enumerate}
\end{theorem}

\begin{corollary} Let $i_0=2$,  $k\geq 3$,  and let $(u_1,...,u_k)$ be the solution of \eqref{4.1}.  Assume at least one of $m_1$ and $m_2$ is nonconstant in the problem \eqref{4.1},   $\mathbf{(G_1)}$ and one of $\mathbf{(F_{2})}$ or \eqref{4.9a} hold. Let $(u_1^*,u_2^*)$ be the positive solution of \eqref{4.14a} obtained by Lemma {\rm\ref{lemma4.7}}. If $a_{21}<L_1$, $a_{12}<L_2$, then
$\dd\lim_{t\to\yy}u_i(x,t)=u_i^*(x)$ for $i=1,2$ and $\dd\lim_{t\to\yy}u_j(x,t)=0$ for $3\leq j\leq k$ uniformly on $\ol\Omega$.
\end{corollary}

\begin{proof}
Thanks to the second result in \eqref{4.15}, we have that $\dd\lim_{t\to\yy}u_j(x,t)=0$ for $3\leq j\leq k$ uniformly on $\ol\Omega$. Then for any $0<\var \ll1$, there exists $T>0$ such that
\bes\label{4.19}
0\leq u_j(x,t)<\var, \ \ \forall\ x\in\ol\Omega,\; t>T,\; 3\leq j\leq k.
\ees
It follows from \eqref{4.1} and \eqref{4.19} that $u_1$ and $u_2$ satisfy
 \begin{equation}\label{4.20}
\begin{cases}
\dd \partial_t u_i\geq d_i\Delta u_i+u_i\bigg(m_i(x)-\sum_{3\leq j\leq k}a_{ij}\var-\sum_{1\leq j\leq 2} a_{ij}u_j\bigg), &x\in \Omega,\; t>T,\; i=1,2,\\[4mm]
\dd \partial_t u_i\leq d_i\Delta u_i+u_i\bigg(m_i(x)-\sum_{1\leq j\leq 2} a_{ij}u_j\bigg), \ \ &x\in \Omega,\; t>T,\;i=1,2.
\end{cases}
 \end{equation}
Denote by $(\bar u_1,\underline u_1,\bar u_2,\underline u_2)$ the solution of
 \begin{equation}\label{4.21}
\begin{cases}
\dd \partial_t \bar u_1=d_1\Delta \bar u_1+\bar u_1\big(m_1(x)-\bar u_1- a_{12}\underline u_2\big), & x\in \Omega,\; t>T,\\[4mm]
\dd \partial_t \underline u_2=d_2\Delta \underline u_2+\underline u_2\bigg(m_2(x)-\sum_{3\leq j\leq k}a_{2j}\var-\underline u_2- a_{21}\bar u_1\bigg), & x\in \Omega,\; t>T,\\[4mm]
\dd \partial_t \underline u_1=d_1\Delta \underline u_1+\underline u_1\bigg(m_1(x)-\sum_{3\leq j\leq k}a_{1j}\var-\underline u_1- a_{12}\bar u_2\bigg), & x\in \Omega,\; t>T,\\[4mm]
\dd \partial_t \bar u_2=d_2\Delta \bar u_2+\bar u_2\big(m_2(x)-\bar u_2- a_{21}\underline u_1\big), & x\in \Omega,\; t>T,\\[2mm]
\dd \partial_\nu\bar u_1=\partial_\nu\underline u_1=\dd \partial_\nu\bar u_2=\partial_\nu\underline u_2=0,& x\in \partial\Omega,\;
 t>T,\\
\bar u_1(x,0)=\underline u_1(x,0)=u_1(x,T),\ \bar u_2(x,0)=\underline u_2(x,0)=u_2(x,T),& x\in \Omega.
 \end{cases}
 \end{equation}
Remember $a_{ii}=1$. Owing to \eqref{4.1}, \eqref{4.20} and \eqref{4.21}, we see that $(\bar u_1,\bar u_2)$ and $(\underline u_1,\underline u_2)$ are { the coupled ordered} upper and lower solutions of the problem \eqref{4.1} with $k=2$ and
\begin{equation}\label{4.22}
\underline u_i(x,t)\leq u_i(x,t)\leq \bar u_i(x,t),\ \ \forall\;x\in\ol\Omega,\;t>0,\; i=1,2.
\end{equation}
Since $a_{21}<L_1$ and $a_{12}<L_2$, we can choose $\var>0$ small enough such that
\begin{equation}\label{4.23}
\begin{split}
a_{21}<L_{d_1,m_1,m_2-\var_{2}},\ a_{12}<L_{d_2,m_2-\var_{2},m_1},\  a_{21}<L_{d_1,m_1-\var_{1},m_2},\ a_{12}<L_{d_2,m_2,m_1-\var_{1}},
\end{split}
\end{equation}
where $\var_{1}=\sum_{3\leq j\leq k}a_{1j}\var$, $\var_{2}=\sum_{3\leq j\leq k}a_{2j}\var$, and
 $$L_{d_1,m_1,m_2-\var_{2}},\ \ L_{d_2,m_2-\var_{2},m_1},\ \  L_{d_1,m_1-\var_{1},m_2}, \ \ L_{d_2,m_2,m_1-\var_{1}}$$
are defined as in \eqref{3.4}.
Then making use of Theorem \ref{th3.1} and \eqref{4.23}, we have
\begin{equation}\label{4.23aa}
\begin{cases}
 \dd\lim_{t\to\yy}\bar u_1(x,t)=\phi_{m_1,m_2-\var_{2}},\ \ \ \lim_{t\to\yy}\underline u_2(x,t)=\psi_{m_1,m_2-\var_{2}},\\[1mm]
\dd\lim_{t\to\yy}\underline u_1(x,t)=\phi_{m_1-\var_{1},m_2},\ \ \
 \lim_{t\to\yy}\bar u_2(x,t)=\psi_{m_1-\var_{1},m_2}
\end{cases}
\end{equation}
uniformly on $\ol\Omega$, where $(\phi_{f,g},\psi_{f,g})=(\phi,\psi)$ is a positive solution of
\begin{equation}\label{4.25}
\begin{cases}
d_1\Delta\phi+\phi\big(f-\phi- a_{12}\psi\big)=0, & x\in \Omega,\\[.5mm]
\dd d_2\Delta\psi+\psi\big(g-\psi-a_{21}\phi\big)=0, & x\in \Omega,\\[1mm]
\dd \partial_\nu\phi=\partial_\nu\psi=0,& x\in\partial\Omega.
\end{cases}
 \end{equation}
Combining \eqref{4.22} with \eqref{4.23aa} we get
 \begin{equation}\label{4.24}
\begin{cases}
 \dd\phi_{m_1-\var_{1},m_2}\leq\liminf_{t\to\yy}u_1(x,t)\leq\limsup_{t\to\yy}u_1(x,t)\leq \phi_{m_1,m_2-\var_{2}},\\[1mm]
 \dd\psi_{m_1,m_2-\var_{2}}\leq\liminf_{t\to\yy}u_2(x,t)\leq\limsup_{t\to\yy}u_2(x,t)\leq \psi_{m_1-\var_{1},m_2},
\end{cases}
\end{equation}
where $\var_{1}=\sum_{3\leq j\leq k}a_{1j}\var$ and $\var_{2}=\sum_{3\leq j\leq k}a_{2j}\var$.  Using the Schauder theory for elliptic equations we can show that $\phi_{m_1-\var_{1},m_2}$, $\psi_{m_1-\var_{1},m_2}$, $\phi_{m_1,m_2-\var_{2}}$ and $\psi_{m_1,m_2-\var_{2}}$ are uniformly bounded for $0<\var\ll 1$ in $C^{2+\alpha}({\ol\Omega})$. Passing to a subsequence of $\var$ if necessary, we may assume
\begin{equation}\label{4.26}
\begin{cases}
 \dd\lim_{\var\to 0^+}\phi_{m_1-\var_{1},m_2}=\phi^{(1)}_{m_1,m_2},\ \ \ \lim_{\var\to 0^+}\psi_{m_1-\var_{1},m_2}=\psi^{(1)}_{m_1,m_2},\\[1mm]
 \dd\lim_{\var\to 0^+}\phi_{m_1,m_2-\var_{2}}=\phi^{(2)}_{m_1,m_2},\ \ \ \lim_{\var\to 0^+}\psi_{m_1,m_2-\var_{2}}=\psi^{(2)}_{m_1,m_2},
\end{cases}
\end{equation}
where $(\phi^{(1)}_{m_1,m_2},\psi^{(1)}_{m_1,m_2})$ and $(\phi^{(2)}_{m_1,m_2},\psi^{(2)}_{m_1,m_2})$ satisfy \eqref{4.25} with $f=m_1$ and $g=m_2$. Since $a_{21}<L_1$ and $a_{12}<L_2$, making use of Theorem \ref{th3.1}, we conclude that $\phi^{(1)}_{m_1,m_2}=\phi^{(2)}_{m_1,m_2}=u_1^*$ and $\psi^{(1)}_{m_1,m_2}=\psi^{(2)}_{m_1,m_2}=u_2^*$. Thus, by  \eqref{4.26},
\bess
\lim_{\var\to 0^+}\phi_{m_1-\var_{1},m_2}=\lim_{\var\to 0^+}\phi_{m_1,m_2-\var_{2}}=u_1^*,\ \ \ \lim_{\var\to 0^+}\psi_{m_1-\var_{1},m_2}=\lim_{\var\to 0^+}\psi_{m_1,m_2-\var_{2}}=u_2^*.
\eess
Combining this with \eqref{4.24},  then the arbitrariness of $\var>0$ yields $\dd\lim_{t\to\yy}u_i(x,t)=u_i^*(x)$ uniformly on $\ol\Omega$ for $i=1,2$. The proof is finished.
\end{proof}

\begin{footnotesize}
	\bibliographystyle{plain}
	\bibliography{NSW}

\begin{thebibliography}{10}

\bibitem{bp1983}
P.~N. Brown.
\newblock Decay to uniform states in competitive systems.
\newblock {\em SIAM J. Math. Anal.}, 14(4):659--673, 1983.

\bibitem{cc1991}
R.~S. Cantrell and C.~Cosner.
\newblock The effects of spatial heterogeneity in population dynamics.
\newblock {\em J. Math. Biol.}, 29(4):315--338, 1991.

\bibitem{cc2003}
R.~S. Cantrell and C.~Cosner.
\newblock {\em Spatial ecology via reaction-diffusion equations}.
\newblock Wiley Series in Mathematical and Computational Biology. John Wiley \&
  Sons, Ltd., Chichester, 2003.

\bibitem{cc1993}
R.~S. Cantrell, C.~Cosner, and V.~Hutson.
\newblock Permanence in ecological systems with spatial heterogeneity.
\newblock {\em Proc. Roy. Soc. Edinburgh Sect. A}, 123(3):533--559, 1993.

\bibitem{dhmp1998}
J.~Dockery, V.~Hutson, K.~Mischaikow, and M.~Pernarowski.
\newblock The evolution of slow dispersal rates: a reaction diffusion model.
\newblock {\em J. Math. Biol.}, 37(1):61--83, 1998.

\bibitem{gbs1977}
B.~S. Goh.
\newblock Global stability in many-species systems.
\newblock {\em The American Naturalist}, 111(977):135--143, 1977.

\bibitem{hw1989}
J.~K. Hale and P.~Waltman.
\newblock Persistence in infinite-dimensional systems.
\newblock {\em SIAM J. Math. Anal.}, 20(2):388--395, 1989.

\bibitem{hn2013i}
X.~Q. He and W.-M. Ni.
\newblock The effects of diffusion and spatial variation in {L}otka-{V}olterra
  competition-diffusion system {I}: {H}eterogeneity vs. homogeneity.
\newblock {\em J. Differential Equations}, 254(2):528--546, 2013.

\bibitem{hn2013ii}
X.~Q. He and W.-M. Ni.
\newblock The effects of diffusion and spatial variation in {L}otka-{V}olterra
  competition-diffusion system {II}: {T}he general case.
\newblock {\em J. Differential Equations}, 254(10):4088--4108, 2013.

\bibitem{hn2016}
X.~Q. He and W.-M. Ni.
\newblock Global dynamics of the {L}otka-{V}olterra competition-diffusion
  system: diffusion and spatial heterogeneity {I}.
\newblock {\em Comm. Pure Appl. Math.}, 69(5):981--1014, 2016.

\bibitem{hn2016i}
X.~Q. He and W.-M. Ni.
\newblock Global dynamics of the {L}otka-{V}olterra competition-diffusion
  system with equal amount of total resources, {II}.
\newblock {\em Calc. Var. Partial Differential Equations}, 55(2):Art. 25, 20,
  2016.

\bibitem{hn2017}
X.~Q. He and W.-M. Ni.
\newblock Global dynamics of the {L}otka-{V}olterra competition-diffusion
  system with equal amount of total resources, {III}.
\newblock {\em Calc. Var. Partial Differential Equations}, 56(5):Art. 132, 26,
  2017.

\bibitem{hees1990}
A.~R.~G. Heesterman.
\newblock {\em Matrices and their roots. A textbook of matrix algebra}.
\newblock World Scientific Publishing Co., Inc., Teaneck, NJ, 1990.

\bibitem{hess1991}
P.~Hess.
\newblock {\em Periodic-parabolic boundary value problems and positivity},
  volume 247 of {\em Pitman Research Notes in Mathematics Series}.
\newblock Longman Scientific \& Technical, Harlow; copublished in the United
  States with John Wiley \& Sons, Inc., New York, 1991.

\bibitem{hsu2005}
S.-B. Hsu.
\newblock A survey of constructing {L}yapunov functions for mathematical models
  in population biology.
\newblock {\em Taiwanese J. Math.}, 9(2):151--173, 2005.

\bibitem{hsu1993}
S.-B. Hsu and P.~Waltman.
\newblock On a system of reaction-diffusion equations arising from competition
  in an unstirred chemostat.
\newblock {\em SIAM J. Appl. Math.}, 53(4):1026--1044, 1993.

\bibitem{hlm2002}
V.~Hutson, Y.~Lou, and K.~Mischaikow.
\newblock Spatial heterogeneity of resources versus {L}otka-{V}olterra
  dynamics.
\newblock {\em J. Differential Equations}, 185(1):97--136, 2002.

\bibitem{hs1992}
V.~Hutson and K.~Schmitt.
\newblock Permanence and the dynamics of biological systems.
\newblock {\em Math. Biosci.}, 111(1):1--71, 1992.

\bibitem{kw2017}
T.~Kuniya and J.~L. Wang.
\newblock Lyapunov functions and global stability for a spatially diffusive
  {SIR} epidemic model.
\newblock {\em Appl. Anal.}, 96(11):1935--1960, 2017.

\bibitem{ln2012}
K.-Y. Lam and W.-M. Ni.
\newblock Uniqueness and complete dynamics in heterogeneous
  competition-diffusion systems.
\newblock {\em SIAM J. Appl. Math.}, 72(6):1695--1712, 2012.

\bibitem{levin}
S.~A. Levin.
\newblock Population dynamic models in heterogeneous environments.
\newblock {\em Annu. Rev. Ecol. Evol. Syst.}, 7(1):287--310, 1976.

\bibitem{ls2010jde}
M.~Y. Li and Z.~S. Shuai.
\newblock Global-stability problem for coupled systems of differential
  equations on networks.
\newblock {\em J. Differential Equations}, 248(1):1--20, 2010.

\bibitem{lou2006}
Y.~Lou.
\newblock On the effects of migration and spatial heterogeneity on single and
  multiple species.
\newblock {\em J. Differential Equations}, 223(2):400--426, 2006.

\bibitem{lzz2019}
Y.~Lou, X.~Q. Zhao, and P.~Zhou.
\newblock Global dynamics of a {L}otka-{V}olterra
  competition-diffusion-advection system in heterogeneous environments.
\newblock {\em J. Math. Pures Appl. (9)}, 121:47--82, 2019.

\bibitem{mc2006cana}
C.~Morales-Rodrigo and A.~Su\'{a}rez.
\newblock Uniqueness of solution for elliptic problems with non-linear boundary
  conditions.
\newblock {\em Comm. Appl. Nonlinear Anal.}, 13(3):69--78, 2006.

\bibitem{nsw2018}
W.~J. Ni, J.~P. Shi, and M.~X. Wang.
\newblock Global stability and pattern formation in a nonlocal diffusive
  {L}otka-{V}olterra competition model.
\newblock {\em J. Differential Equations}, 264(11):6891--6932, 2018.

\bibitem{pcv1992}
C.~V. Pao.
\newblock {\em Nonlinear parabolic and elliptic equations}.
\newblock Plenum Press, New York, 1992.

\bibitem{pao2007na}
C.~V. Pao.
\newblock Quasilinear parabolic and elliptic equations with nonlinear boundary
  conditions.
\newblock {\em Nonlinear Anal.}, 66(3):639--662, 2007.

\bibitem{Smith}
H.~L. Smith.
\newblock {\em Monotone dynamical systems: an introduction to the theory of
  competitive and cooperative systems}, volume~41 of {\em Mathematical Surveys
  and Monographs}.
\newblock American Mathematical Society, Providence, RI, 1995.

\bibitem{tilman1994competition}
D.~Tilman.
\newblock Competition and biodiversity in spatially structured habitats.
\newblock {\em Ecology}, 75(1):2--16, 1994.

\bibitem{wmx2016}
M.~X. Wang.
\newblock A diffusive logistic equation with a free boundary and sign-changing
  coefficient in time-periodic environment.
\newblock {\em J. Funct. Anal.}, 270(2):483--508, 2016.

\bibitem{wmx2018}
M.~X. Wang.
\newblock Note on the {L}yapunov functional method.
\newblock {\em Appl. Math. Lett.}, 75:102--107, 2018.

\bibitem{WZhang18}
M.~X. Wang and Y.~Zhang.
\newblock Dynamics for a diffusive prey-predator model with different free
  boundaries.
\newblock {\em J. Differential Equations}, 264(5):3527--3558, 2018.

\bibitem{zx2018jfa}
P.~Zhou and D.-M. Xiao.
\newblock Global dynamics of a classical {L}otka-{V}olterra
  competition-diffusion-advection system.
\newblock {\em J. Funct. Anal.}, 275(2):356--380, 2018.

\end{thebibliography}
\end{footnotesize}
\end{document}